\documentclass[12pt]{amsart} 
\usepackage[mathscr]{eucal} 
\usepackage{amsmath,amsfonts} 
\usepackage{amssymb}
\usepackage{color}
\usepackage{graphicx}
\parskip=\smallskipamount 
\newtheorem{theorem}{Theorem}[section] 
\newtheorem{proposition}[theorem]{Proposition} 
\newtheorem{corollary}[theorem]{Corollary} 
\newtheorem{lemma}[theorem]{Lemma} 
\theoremstyle{definition} 
\newtheorem{definition}[theorem]{Definition} 
\newtheorem{example}[theorem]{Example}

\newtheorem{remark}[theorem]{Remark} 

\makeatletter 
\@addtoreset{equation}{section} 
\makeatother 
\renewcommand\theequation{\arabic{section}.\arabic{equation}} 
 
\newcommand{\CC}{{\mathbb C}} 
\newcommand{\NN}{{\mathbb N}} 
\newcommand{\QQ}{{\mathbb Q}}

\newcommand{\RR}{{\mathbb R}} 
\newcommand{\FF}{{\mathbb F}}

\newcommand{\cB}{{\mathcal B}} 
\newcommand{\cC}{{\mathcal C}} 
 
\newcommand{\cE}{{\mathcal E}} 
\newcommand{\cF}{{\mathcal F}} 
\newcommand{\cG}{{\mathcal G}} 
\newcommand{\cH}{{\mathcal H}} 

\newcommand{\cJ}{{\mathcal J}} 
\newcommand{\cK}{{\mathcal K}} 
\newcommand{\cL}{{\mathcal L}} 
 
\newcommand{\cN}{{\mathcal N}} 
\newcommand{\cP}{{\mathcal P}} 

\newcommand{\cR}{{\mathcal R}} 
\newcommand{\cS}{{\mathcal S}}

\newcommand{\bx}{\mathbf{x}}

\newcommand{\Ra}{\Rightarrow} 
 
\newcommand{\ran}{\operatorname{Ran}} 
 
\newcommand{\ra}{\rightarrow} 
\newcommand{\ol}{\overline}

\newcommand{\tr}{\operatorname{tr}} 
\let\phi=\varphi 
\newcommand{\iac}{\mathrm{i}} 
\renewcommand{\ker}{\operatorname{Null}}

\newcommand{\de}{\mathrm{d}}

\newcommand{\supp}{\operatorname{supp}}

\newcommand{\Span}{\operatorname{Span}}
\newcommand{\card}{\operatorname{card}}
\newcommand{\rank}{\operatorname{rank}}
\renewcommand{\Re}{\operatorname{Re}}
\newcommand{\Diag}{\operatorname{Diag}}

\newcommand{\nr}[1]{\vspace{0.1ex}\noindent\hspace*{12mm}\llap{\textup{(#1)}}} 
 
\begin{document} 
\title[Mercer's Theorem]{An Operator Theoretical Approach to \\
Mercer's Theorem}
\author{Aurelian Gheondea}
\address{Institutul de Matematic\u a al Academiei Rom\^ane, Calea Grivi\c tei 21,  010702 Bucure\c sti, Rom\^ania \emph{and}
Department of Mathematics, Bilkent University, 06800 Bilkent, Ankara, Turkey} 
\email{A.Gheondea@imar.ro \textrm{and} aurelian@fen.bilkent.edu.tr} 

\begin{abstract} This is a survey article on Mercer's Theorem in its most general form 
and its relations with the theory of reproducing 
kernel Hilbert spaces and the spectral theory of compact operators. We provide a modern introduction to 
the basics of the theory of reproducing kernel Hilbert spaces, an overview on Weyl's kernel and the Gaussian 
kernels, and finally an approach to Mercer's Theorem within the theory of reproducing kernel Hilbert spaces and the spectral 
theory of integral operators. 
This approach is reverse to the known approaches to Mercer's Theorem and sheds some light on the intricate 
relations between different domains in analysis.
\end{abstract}



\maketitle

\section{Introduction}

The idea of writing this manuscript occurred to me during 
the presentations I have made in the Seminar of Operator Theory at the 
Institute of Mathematics of the Romanian Academy in Bucharest, following the survey paper of F.~Cucker and S.~Smale 
\cite{CuckerSmale} on the mathematical foundations of machine learning. To be more precise, this refers to 
Chapter III in that paper,  which deals with the generalisation of 
the celebrated Mercer's Theorem \cite{Mercer}, to the case when the base set $X$ is a compact metric space 
endowed with a finite Borel measure $\nu$ with full support,
and its consequences for reproducing kernel Hilbert spaces induced by continuous positive semidefinite kernels 
on $X$. On the one hand, the goals of the paper \cite{CuckerSmale} are strongly related to machine 
learning and kernel methods have been used successfully for some good time, for example, see the monographs of I.~Steinwart 
and A.~Christman \cite{SteinwartChristman} and that of  B. Sch\"olkopf and A.J. Smola \cite{ScholkopfSmola}.  
On the other hand, recently there is a new direction of research in machine learning and 
dynamical systems, for example see the recent survey paper of L.~Rosasco \cite{Rosasco}, for which reproducing kernel Hilbert 
spaces make an essential tool. The relations between these two directions are quite strong, in view of the fact that the main idea of 
the kernel method is linearisation of certain nonlinear problems. We think that having a careful 
presentation, with most of the technical details clarified from the mathematics point of view, will be helpful for both communities, 
mathematicians and engineers working in machine learning. An outline of the results and 
the flow of proofs on Mercer's Theorem in this article were presented in the Machine Learning and Dynamical Systems Seminar, 
organised by Boumediene Hamzi as an online seminar at the Alan Turing Institute. 

Also, we thought that it is more illuminating to put Mercer's Theorem into a different perspective, when compared with other 
available resources. For this reason, first we make a careful presentation of the basics of reproducing kernel Hilbert space in the 
full generality, that is, when the set $X$ bears no other structure on it, 
and prove the most important results that are required for our enterprise. 
In this respect, we approach Moore's Theorem of the existence of reproducing kernel 
Hilbert spaces associated to a given positive semidefinite kernel through Kolmogorov's Theorem on 
the linearisation of this kind of kernels, which shortcuts a significant part of the original proof, and we make a 
clear connection with the theory of operator ranges. In view of a 
clear perspective on Mercer's Theorem, the recovering of kernels from orthonormal bases of the corresponding 
reproducing kernel Hilbert space plays an important role and, as we will show, clarifies an important aspect in the 
proof and the interpretation of Mercer's Theorem. Also, we briefly make the connection of reproducing kernel 
Hilbert spaces with operator ranges and prove Aronszajn's Inclusion Theorem by exploiting this point of view.

As a sequel of the basics of the theory of reproducing kernel Hilbert spaces, we overview two of the 
most useful kernels in machine learning, 
Weyl's kernel and the Gaussian kernel. For the proofs of many more results we recommend 
the article of H.Q.~Minh \cite{Minh}, as well as the article of I.~Steinwart, 
D.~Hush, C.~Scovel \cite{SteinwartHushScovel} for previous results on this issue.

Mercer's Theorem was obtained in 1909 and it was motivated by the results of D.~Hilbert and his student 
E.~Schmidt on integral equations, which triggered what is now called Hilbert-Schmidt operators and lead to
one of the most useful ideal of compact operators on a Hilbert space. J.~Mercer was the first who singled out the 
concept of positive semidefinite kernel in the modern understanding and our aim is to put his celebrated theorem 
into an operator theoretical perspective and shed more light on the deep relation of it with the theory of 
reproducing kernel Hilbert spaces. Originally, this theorem was stated for the case of $X$ a compact interval of 
the real line and $\nu$ the Lebesgue measure. In the classical texts, such as 
F.~Riesz and B.~Sz.-Nagy \cite{RieszNagy} 
and I.~Gohberg, S.~Goldberg, and M.A.~Kaashoek \cite{GohbergGoldberg}, the lines of the original 
proof are followed. Generalisations of this theorem can be found in F.~Cucker and D.-X.~Zhou \cite{CuckerZhou} 
and in the two most important monographs on the 
theory of reproducing kernel Hilbert spaces, which are V.I.~Paulsen and M.~Raghupathi 
\cite{PaulsenRaghupathi} and S.~Saitoh and Y.~Sawano \cite{SaitohSawano}, to the case when $X$ is a 
compact metric space and $\nu$ is a finite Borel measure on $X$, but each of the proofs in the cited 
monographs has flaws. Paulsen-Raghupathi misses the most important conclusion in the Mercer's Theorem, 
which is the uniform convergence of the series, while Cucker-Zhou and 
Saitoh-Sawano miss the pointwise convergence, which is  essential in applying Dini's Theorem. 

Another important reason why we have chosen this operator theoretical approach in order 
to prove the general version of 
the Mercer's Theorem is that, when J.~Mercer obtained its classical version, his motivation was driven by the 
Hilbert-Schmidt theory of integral operators and, at that time, 
the concept of reproducing kernel Hilbert space was 
not available, although the idea, but not the concept, first appeared two years before at S.~Zaremba 
\cite{Zaremba}.
Except for a short note on positive semidefinite kernels by E.H.~Moore \cite{Moore1}, these ideas have not 
being used until the dissertations of three Berlin
mathematicians: in 1921 by G.~Szeg\"o \cite{Szego} and in 1922 by S.~Bergman \cite{Bergman} and 
S.~Bochner \cite{Bochner}.
In particular, S.~Bergman introduced reproducing kernels in one and several
variables for the class of harmonic and analytic functions and he called
them ”kernel functions”. These ideas have been used 
later by E.H.~Moore  \cite{Moore} and their 
abstract conceptualisation was performed by N.~Aronszajn in \cite{Aronszajn1} and \cite{Aronszajn2}. A 
completely different point 
of view was taken by L.~Schwartz in 1964 \cite{Schwartz}, who considered Hilbert spaces continuously included 
in quasi-complete locally convex spaces, and it turned out that this theory is equivalent with Aronszajn's theory. 
More than that, 
Schwartz considered not only positive semidefinite kernels but also Hermitian kernels and, 
correspondingly, he obtained a theory of reproducing kernel Krein spaces, which are generalisations of Hilbert 
spaces to the case 
when the inner product is not positive definite, but only Hermitian, and with certain geometric-topological 
properties. A survey on reproducing kernel Krein spaces can be found in \cite{GhOT2015}, while a recent
monograph on Krein spaces and their linear operators  is  \cite{Gheondea}, where the role of Hermitian, but not 
positive semidefinite kernels, can be seen for Carath\'eodory, Schur, and Nevanlinna kernels. There is an 
interest in learning theory with kernels 
that are not positive semidefinite but only symmetric, as seen at C.S.~Ong at al. \cite{Ong} and D.~Oglic 
and T.~G\"artner \cite{OglicGartner}.

However, although later it was realised that Mercer's Theorem has important consequences for the theory of 
reproducing kernel 
Hilbert spaces, a clear explanation on why it is so closely related to  that theory was not provided. For 
this reason, we take a reverse order in the approach of Mercer's Theorem, through the spectral theory of 
compact operators in Hilbert spaces and the theory of reproducing kernel Hilbert spaces. More precisely, we 
show that Mercer's 
Theorem is a special case of the more general theorem of recovering the reproducing kernel from an 
orthonormal basis of the 
reproducing kernel Hilbert space. This approach is used in a different setting
by C.~Carmeli, E.~De Vito, A.~Toigo \cite{Carmeli} for a generalisation of the Mercer's Theorem to the operator 
valued kernels. An approach closer to our presentation can be found also in the monograph of 
F.~Cucker and D.-X. Zhou \cite{CuckerZhou}, but some of the technical details differ considerably. Also, our 
journey through the preliminary results for Mercer's Theorem passes through the operator range representation 
of reproducing kernel Hilbert spaces associated to Mercer kernels.

The reader of this article is supposed to have 
a good command on operator theory on Hilbert spaces. We recommend
a few textbooks and monographs where most of the basic concepts and results can be found: 
F.~Riesz, B.~Sz.-Nagy \cite{RieszNagy}
M.S.~Birman, M.Z.~Solomjak \cite{BirmanSolomjak},
J.B.~Conway \cite{Conway1} and \cite{Conway}, I.~Gohberg, S.~Goldberg, M.A.~Kaashoek 
\cite{GohbergGoldberg}, A.W.~Knapp \cite{KnappARA}. The basics of real analysis that we use here can be 
found in A.W.~Knapp \cite{KnappBRA}, while the measure theory can be found at H.S.~Bear \cite{Baer} and, for 
more advanced topics, at H.~Geiss and S.~Geiss \cite{Geiss2025}, for example.
For the readers' convenience we included
an appendix, in which we review most of the definitions and results that are needed in 
connection to compact, trace-class, and Hilbert-Schmidt operators.\medskip

\textbf{Acknowledgements:} Thanks are due to the participants of the Operator Theory Seminar in the Institute of Mathematics of 
the Romanian Academy in Bucharest for their questions, comments, and corrections 
during the presentations of these notes and especially to Alexandru Must\u a\c tea for 
a very thorough reading of the manuscript and a very detailed list of corrections and observations. An outline of the results and 
the flow of proofs on Mercer's Theorem in this article was presented in the Machine Learning and Dynamical Systems Seminar 
organised by Boumediene Hamzi as an online seminar at the Alan Turing Institute. 
The author expresses his thanks for this invitation and for the 
questions and discussions that followed the two presentations. 

\section{Reproducing Kernel Hilbert Spaces --- The Basics}\label{s:rkhs}

In this section we present the basics for the theory of reproducing kernel Hilbert spaces following mostly the 
conceptualisation of N.~Aronszajn \cite{Aronszajn1} and \cite{Aronszajn2} but in a modern presentation. Use of N.~Kolmogorov's 
Theorem \cite{Kolmogorov1}, \cite{Kolmogorov2} shortcuts some intricate proofs. There is another parallel 
conceptualisation of this theory of L.~Schwartz \cite{Schwartz}, but this is only touched upon in
Remark~\ref{r:schwartz} and Remark~\ref{r:opran}.

\subsection{Positive Semidefinite Kernels}
Let $X$ be a nonempty set and consider a \emph{kernel} $K$ on $X$ and valued in $\FF$, that is,
$K\colon X\times X\ra \FF$, where $\FF$ is either $\CC$ or $\RR$. By definition, 
$K$ is \emph{Hermitian} (if $\FF=\CC$) or \emph{symmetric} (if $\FF=\RR)$ if
\begin{equation*}
K(y,x)=\overline{K(x,y)},\quad x,y\in X.
\end{equation*}
Also, $K$ is \emph{positive semidefinite} if for all $n\in\NN$, all $x_1,\ldots,x_n\in X$, and all $\alpha_1,\ldots,\alpha_n\in\FF$ we have
\begin{equation*}
\sum_{i,j=1}^n \overline{\alpha_i}\alpha_j K(x_i,x_j)\geq 0.
\end{equation*}
Note that this means that for any $n\in\NN$ and any $x_1,\ldots,x_n\in X$, the $n\times n$ matrix 
$[K(x_i,x_j)]_{i,j=1}^n$ is positive as a linear operator $\FF^n\ra\FF^n$. 

\begin{remark} In the case $\FF=\CC$, positive 
semidefiniteness of $K$ implies that $K$ is Hermitian. Indeed, first, letting $n=1$ it follows that $K(x,x)\geq 0$ for all $x\in X$. 
Then, for arbitrary $x,y\in X$ and $\alpha,\beta\in \CC$ we have
\begin{equation*}
|\alpha|^2 K(x,x)+\alpha\ol{\beta}K(x,y)+\ol{\alpha}\beta K(y,x)+|\beta|^2K(y,y)\geq 0,
\end{equation*}
hence
\begin{equation*}
\alpha\ol{\beta}K(x,y)+\ol{\alpha}\beta K(y,x)\in\RR.
\end{equation*}
Replacing $\beta$ with $\iac \beta$ from here we get
\begin{equation*}
\alpha\ol{\beta}K(x,y)-\ol{\alpha}\beta K(y,x)\in\iac\RR.
\end{equation*}
Letting $\alpha=\beta=1$, from the last two properties we get $K(x,y)=\ol{K(y,x)}$.

If $\FF=\RR$ the positive definiteness of $K$ does not 
imply that it is symmetric. A simple example is obtained if $X=\{1,2\}$ and $K$ is defined by the $2\times 2$ real 
matrix
\begin{equation*}
K=\begin{bmatrix} 1 & 1 \\ 0 & 1\end{bmatrix}.
\end{equation*}
\end{remark}

\begin{lemma}[Schwarz's Inequality]\label{l:schwarz}
If $K$ is a Hermitian/symmetric and positive semidefinite kernel, then
\begin{equation*}
|K(x,y)|^2\leq K(x,x) K(y,y),\quad x,y\in X.
\end{equation*}
\end{lemma}

\begin{proof} For arbitrary $x,y\in X$ we consider the $2\times 2$ matrix
\begin{equation*}
A=\begin{bmatrix} K(x,x) & K(x,y) \\ K(y,x) & K(y,y)\end{bmatrix},
\end{equation*}
which is Hermitian/symmetric and positive, hence its determinant is
\begin{equation*}
\det (A)=K(x,x)K(y,y)-|K(x,y)|^2\geq 0,
\end{equation*}
which is exactly the required inequality.
\end{proof}

\begin{theorem}[N.~Kolmogorov \cite{Kolmogorov1}, \cite{Kolmogorov2}]\label{t:kolmo}
Let $K\colon X\times X\ra\FF$ be a kernel. 

\nr{a} The following assertions are equivalent.
\begin{itemize}
\item[(1)] $K$ is a (symmetric) positive semidefinite kernel.
\item[(2)] There exists a pair $(\cH,\Phi)$, called a \emph{linearisation of $K$}, 
where $\cH$ is a Hilbert space and $\Phi\colon X\ra \cH$ is a map, 
such that
\begin{equation*}
K(x,y)=\langle \Phi(y),\Phi(x)\rangle_{\cH},\quad x,y\in X.
\end{equation*}
\end{itemize}

\nr{b} If it exists, the linearisation 
$(\cH,\Phi)$ can always be chosen \emph{minimal}, in the sense that $\Phi(X)$ is a 
total subset of $\cH$.

\nr{c} If $(\cH_1,\Phi_1)$ and $(\cH_2,\Phi_2)$ are two minimal linearisations of $K$ then there exists uniquely a unitary operator $U\colon \cH_1\ra\cH_2$ such that $U\Phi_1(x)=\Phi_2(x)$ for all $x\in X$.
\end{theorem}

In the literature, a linearisation $(\cH,\Phi)$ was also called a \emph{Kolmogorov decomposition}, 
see T.~Constantinescu \cite{Constantinescu}, while in the 
machine learning community it is called a \emph{feature pair}, made up by a \emph{feature space} $\cH$ and
a \emph{feature map} $\Phi$.

Before approaching the proof of the Kolmogorov's Theorem a few notations and facts are in order. First, we consider the vector space of all functions on $X$
\begin{equation}\label{e:fex}
\cF_\FF(X):=\{f\colon X\ra \FF\mid f\mbox{ function }\},
\end{equation}
and its subspace of functions with finite support
\begin{equation}\label{e:fexo}
\cF_\FF^0 (X):=\{f\in \cF_\FF(X)\mid \supp(f)\mbox{ finite }\}.
\end{equation}
For each $x\in X$ we consider the seminorm $p_x$ on $\cF_\FF(X)$ defined by
\begin{equation}\label{e:pex}
p_x(f):=|f(x)|,\quad f\in \cF_\FF(X).
\end{equation}
It is easy to see that the vector space $\cF_\FF(X)$ endowed with the calibration $\{p_x\mid x\in X\}$ becomes a 
complete locally convex space and the underlying topology is nothing else than the topology of pointwise 
convergence.

The vector space $\cF_\FF^0(X)$ has a linear basis made up by $\delta_x$, for all $x\in X$,
with $\delta_x(y)=1$ if $x=y$ and $0$ elsewhere. Also, it has a natural inner product 
$\langle\cdot,\cdot\rangle_0$ defined by
\begin{equation}\label{e:ipo}
\langle f,g\rangle_0:=\sum_{x\in X}f(x)\overline{g(x)},\quad f,g\in \cF_\FF^0(X).
\end{equation}
However, an important observation is that the pairing $\langle f,g\rangle_0$ makes sense if at least one of 
$f$ or $g$ has finite support, the other could be any function in $\cF_\FF(X)$.

Let now $K\colon X\times X\ra \FF$ be a kernel. A \emph{convolution operator} $C_K\colon \cF_\FF^0(X)\ra \cF_\FF(X)$ can be defined by
\begin{equation}\label{e:ceka}
(C_Kf)(x):=\sum_{y\in X} K(x,y)f(y),\quad f\in \cF_0(X).
\end{equation}
In particular, from here it follows that
\begin{equation*}
C_K\delta_x=K_x,\quad x\in X,
\end{equation*}
where, for any $x\in X$ we let $K_x\in \cF_\FF(X)$ be defined by
\begin{equation}\label{e:kex}
K_x(y):=K(y,x),\quad y\in X.
\end{equation}
Observe that the mapping $K\mapsto C_K$ is one-to-one and
that, from \eqref{e:ceka}, it follows
\begin{equation*}
(C_Kf)(x)=\sum_{y\in X}K_y(x)f(y), \quad f\in \cF_0(X),
\end{equation*}
from which we get
\begin{equation*}
\ran(C_K)=\Span\{K_x\mid x\in X\}.
\end{equation*}

Recall that $K^*(x,y)=\overline{K(y,x)}$ is the \emph{adjoint kernel} of $K$. Then, the following holds
\begin{equation*}
\langle C_Kf,g\rangle_0=\langle f,C_{K^*}g\rangle_0,\quad f,g\in \cF_\FF^0(X).
\end{equation*}
In addition, $K$ is positive semidefinite if and only if $C_K\geq 0$, in the sense that 
$\langle C_Kf,f\rangle_0\geq 0$ for all $f\in\cF_\FF^0(X)$. This can be seen by letting, for arbitrary 
$f\in\cF_\FF^0(X)$, $\supp(f)=\{x_1,\ldots,x_n\}$ and $f(x_i)=\alpha_i$ for all $i=1,\ldots,n$, and then 
observing that
\begin{equation*}
\langle C_Kf,f\rangle_0=\sum_{i,j=1}^n \overline{\alpha_i}\alpha_j K(x_i,x_j).
\end{equation*}

\begin{proof}[Proof of Theorem~\ref{t:kolmo}] (a) (1)$\Ra$(2). On $\cF_\FF^0(X)$ define the pairing $\langle \cdot,\cdot\rangle_K$ by
\begin{equation*}
\langle f,g\rangle_K:=\langle C_Kf,g\rangle_0=\sum_{x,y\in X}K(x,y)f(y)\overline{g(x)},\quad f,g\in \cF_\FF^0(X).
\end{equation*}
Note that $\langle\cdot,\cdot\rangle_K$ is a pre-inner product: Hermitian/symmetric because $K$ is this way and positive semidefinite because $K$ is positive semidefinite. In particular, the Schwarz Inequality holds
\begin{equation*}
|\langle f,g\rangle_K|^2\leq \langle f,f\rangle_K \langle g,g\rangle_K,\quad f,g\in\cF_\FF^0(X).
\end{equation*}
As a consequence, we have that
\begin{equation}\label{e:neka}
\cN_K:=\{f\in\cF_\FF^0(X)\mid \langle f,f\rangle_K=0\}=\ker(C_K),
\end{equation}
is a vector subspace of $\cF^0_\FF(X)$ and, letting $\pi_K\colon \cF^0_\FF(X)\ra \cF^0_\FF(X)/\cN_K$ be the canonical 
projection, let $\cH$ be a completion of $\cF^0_\FF(X)/\cN_K$ endowed with the induced inner product 
$\langle\cdot,\cdot\rangle_K$ to a Hilbert space.

Let $\Phi\colon X\ra\cH$ be defined by
\begin{equation*}
\Phi(x):=\pi_K(\delta_x),\quad x\in X.
\end{equation*}
Then, for arbitrary $x,y\in X$ we have
\begin{align*}
\langle \Phi(x),\Phi(y)\rangle_\cH & = \langle \pi_K(\delta_x),\pi_K(\delta_y)\rangle_\cH = \langle \delta_x,\delta_y\rangle_K \\
& = \langle C_K\delta_x,\delta_y\rangle_0 = \sum_{z\in X}(C_K\delta_x)(z)\delta_y(z) \\
&= (C_K\delta_x)(y) =\sum_{z\in X}K(y,z)\delta_x(z)=K(y,x).
\end{align*}
This concludes the proof that the pair $(\cH,\Phi)$ is a linearisation of $K$

(2)$\Ra$(1). For arbitrary $n\in \NN$ and $x_1,\ldots,x_n\in X$ we have
\begin{equation*}
[K(x_i,x_j)]_{i,j=1}^n = [\langle \Phi(x_j),\Phi(x_i)\rangle_\cH]_{i,j=1}^n\geq 0,
\end{equation*}
because the latter is the Gram matrix of the vectors $\Phi(x_1),\ldots,\Phi(x_n)$. More explicitly, for any $\alpha_1,\ldots,\alpha_n\in\FF$ we have
\begin{align*}
\sum_{i,j=1}^n \overline{\alpha_i}\alpha_j \langle \Phi(x_j),\Phi(x_i)\rangle_\cH & = \langle \sum_{j=1}^n \alpha_j \Phi(x_j),\sum_{i=1}^n \alpha_i\Phi_i\rangle_\cH = \biggl\| \sum_{j=1}^n \alpha_j\Phi(x_j)\biggr\|_\cH^2\geq 0.
\end{align*} 

(b) Note that for the linearisation $(\cH,\Phi)$ constructed during the proof of the implication (1)$\Ra$(2) we have
$\Phi(X)=\{\pi_K(\delta_x)\mid x\in X\}$ which is total in $\cH$ because $\{\delta_x\mid x\in X\}$ is a linear basis 
of $\cF_\FF^0(X)$. 

(c) Let $m,n\in\NN$ be arbitrary and consider arbitrary $x_1,\ldots,x_m,y_1,\ldots,y_n\in X$ and $\alpha_1,\ldots,\alpha_m,\beta_1,\ldots,\beta_n\in\FF$. Then,
\begin{align*}
\langle \sum_{j=1}^m \alpha_j\Phi_1(x_j),\sum_{i=1}^n\beta_i \Phi_1(y_i)\rangle_{\cH_1} & =
\sum_{j=1}^m\sum_{i=1}^n \alpha_j\ol{\beta_i}\langle \Phi_1(x_j),\Phi_1(y_i)\rangle_{\cH_1} \\
\intertext{and, taking into account that $\langle \Phi_1(x_j),\Phi_1(y_i)\rangle_{\cH_1}=K(y_i,x_j)=
\langle \Phi_2(x_j),\Phi_2(y_i)\rangle_{\cH_1}$, for all $j=1,\ldots,m$ and all $i=1,\ldots,n$, this equals}
& = \sum_{j=1}^m\sum_{i=1}^n \alpha_j\ol{\beta_i}\langle \Phi_2(x_j),\Phi_2(y_i)\rangle_{\cH_2} \\
& = \langle \sum_{j=1}^m \alpha_j\Phi_2(x_j),\sum_{i=1}^n\beta_i \Phi_2(y_i)\rangle_{\cH_2},
\end{align*}
which shows that, letting the operator $U\colon \Span\Phi_1(X)\ra\Span\Phi_2(X)$ be defined by
\begin{equation*}
U(\sum_{j=1}^m\alpha_j\Phi_1(x_j))= \sum_{j=1}^n \alpha_j\Phi_2(x_j),
\end{equation*}
this correctly defines a linear operator, which is isometric, is densely defined, and has dense range, hence it can 
be uniquely extended to a unitary operator $U\colon \cH_1\ra\cH_2$ such that $U\Phi_1(x)=\Phi_2(x)$ for all 
$x\in X$.
\end{proof}

\subsection{Reproducing Kernels}
\begin{definition}\label{d:rkhs} 
Let $X$ be a nonempty set. A set $\cH$ is called a \emph{reproducing kernel Hilbert space} on $X$ if the 
following conditions hold.
\begin{itemize}
\item[(1)] $\cH\subseteq \cF_\FF(X)$ as a vector space.
\item[(2)] $(\cH,\langle\cdot,\cdot\rangle_\cH)$ is a Hilbert space.
\item[(3)] There exists a kernel $K\colon X\times X\ra\FF$ such that, with notation as in \eqref{e:kex}, we
have $K_x\in \cH$ for all $x\in X$.
\item[(4)] For each $f\in \cH$ and each $x\in X$ we have $f(x)=\langle f,K_x\rangle_\cH$.
\end{itemize}
\end{definition}

Property (4) is called the \emph{reproducing property} and the kernel $K$ is called a \emph{reproducing kernel}.

\begin{proposition}\label{p:rkhs}
Let $\cH$ and $K$ be as in Definition~\ref{d:rkhs}. Then,
\begin{itemize}
\item[(a)] $K(x,y)=\langle K_y,K_x\rangle_\cH$ for all $x,y\in X$.
\item[(b)] $K$ is Hermitian/symmetric and positive semidefinite.
\item[(c)] $\{K_x\mid x\in X\}$ is total in $\cH$.
\item[(d)] $K$ is uniquely associated to $\cH$.
\end{itemize}
\end{proposition}

\begin{proof} (a) This is straightforward from the reproducing property and the 
definition of $K_x$ as in \eqref{e:kex}: 
\begin{equation}\langle K_y,K_x\rangle_\cH=K_y(x)=K(x,y),\quad x,y\in X.\end{equation}

(b) First, by (a) we have
\begin{equation*}
K(x,y)=\langle K_y,K_x\rangle_\cH=\overline{\langle K_x,K_y\rangle_\cH}=\overline{K(y,x)},\quad x,y\in X.
\end{equation*}
Then, for arbitrary $n\in\NN$, $x_1,\ldots,x_n\in X$ and $\alpha_1,\ldots,\alpha_n\in\FF$, we have
\begin{align*}
\sum_{i,j=1}^n \ol{\alpha_i}\alpha_j K(x_i,x_j) & = \sum_{i,j=1}^n \ol{\alpha_i}\alpha_j \langle K_{x_j},K_{x_i}\rangle_\cH \\
& = \langle \sum_{j=1}^n \alpha_j K_{x_j},\sum_{i=1}^n \alpha_i K_{x_i}\rangle_\cH \\
& = \| \sum_{j=1}^n \alpha_j K_{x_j}\|_\cH^2\geq 0.
\end{align*}

(c) Let $f\in \cH$ be such that $f\perp K_x$ for all $x\in X$. By the reproducing property, for all $x\in X$ we have
$f(x)=\langle f,K_x\rangle_\cH=0$, hence $f=0$.

(d) Let $K$ and $L$ be two kernels associated to the same reproducing kernel Hilbert space $\cH$. Then, for all $x,y\in X$, using the reproducing property for $L$ and $K$, we have
\begin{equation*}
K(x,y)=K_y(x)=\langle K_y,L_x\rangle_\cH=\ol{L_x(y)}=\ol{L(y,x)}=L(x,y).\qedhere
\end{equation*}
\end{proof}

\begin{corollary} If $\cH$ is a reproducing kernel Hilbert space with reproducing kernel $K$ then the pair
$(\cH,K_\cdot)$, where $K_\cdot$ is the map $X\ni x \mapsto K_x\in\cH$, is a minimal linearisation of $K$.
\end{corollary}

For the following proposition, recall that $\cF_\FF(X)$ is a complete locally convex space with the calibration 
defined as in \eqref{e:pex} which defines the topology of pointwise convergence.

\begin{proposition} \label{p:conte}
Let $\cH$ be a Hilbert space included in $\cF_\FF(X)$ as a 
vector space. The following assertions are equivalent.
\begin{itemize}
\item[(i)] $\cH$ is a reproducing kernel Hilbert space on $X$.
\item[(ii)] $\cH\hookrightarrow \cF_\FF(X)$, that is, $\cH$ is continuously included in $\cF_\FF(X)$.
\item[(iii)] For each $x\in X$ the functional of evaluation at $x$, $\mathrm{Ev}_x\colon \cH\ra\FF$ defined by
$\cH\ni h\mapsto \mathrm{Ev}_x(h)=h(x)\in\FF$, is continuous.
\end{itemize}
\end{proposition}

\begin{proof} (i)$\Ra$(ii). If $\cH$ is a reproducing kernel Hilbert space on $X$ then 
by the reproducing property, given $x\in X$ arbitrary, for each $f\in\cH$ we have 
$f(x)=\langle f,K_x\rangle_\cH$ hence, with notation as in \eqref{e:pex}, we have
\begin{equation*}
p_x(f)=|f(x)|=|\langle f,K_x\rangle_\cH|\leq \| K_x\|_\cH\, \|f\|_\cH,
\end{equation*} 
hence $\cH\hookrightarrow \cF_\FF(X)$.

(ii)$\Ra$(iii). Assume $\cH\hookrightarrow \cF_\FF(X)$ and let $x\in X$ be arbitrary. Then, for each sequence $(h_n)_n$ of functions in $\cH$ that converges to $0$ with respect to the norm $\|\cdot\|_\cH$, it follows that
\begin{equation*}
|\mathrm{Ev}_x(h_n)|=|h_n(x)|=p_x(h_n)\ra 0\mbox{ as }n\ra\infty,
\end{equation*}
hence the linear functional $\mathrm{Ev}_x$ is continuous on $\cH$.

(iii)$\Ra$(i). Assume that for each $x\in X$ the linear functional $\mathrm{Ev}_x\colon \cH\ra\FF$ is continuous. 
Then, by Riesz's Lemma, there exists a unique vector $\zeta_x\in\cH$ such that
\begin{equation*} h(x)=\mathrm{Ev}_x(h)=\langle h,\zeta_x\rangle_\cH,\quad h\in\cH.
\end{equation*}
Define the kernel $K(x,y):=\langle \zeta_y,\zeta_x\rangle_\cH$, for all $x,y\in \cH$. We show that $\cH$ is 
a reproducing kernel Hilbert space on $X$ with reproducing kernel $K$. In order to see this, with respect to 
Definition~\ref{d:rkhs}, the properties (1) and (2) are clearly satisfied. 

For arbitrary $x,y\in X$, $\zeta_x\in\cH$ by construction and then
\begin{equation*}
\zeta_x(y)=\langle \zeta_x,\zeta_y\rangle_\cH=K(y,x)=K_x(y),
\end{equation*}
which proves that $K_x=\zeta_x\in\cH$ for all $x\in X$, hence the property (3) is satisfied as well.
Finally, for arbitrary $f\in\cH$ and $x\in X$, by the definition of $\zeta_x$ we have
\begin{equation*}
f(x)=\langle f,\zeta_x\rangle_\cH=\langle f,K_x\rangle_\cH,
\end{equation*}
hence property (4) is satisfied as well.
\end{proof}

\begin{remark}\label{r:ls} The characterisation of reproducing kernel Hilbert spaces $\cH$ on a set $X$ by 
the fact that $\cH\hookrightarrow \cF_\FF(X)$ makes the connection with a parallel theory 
developed by L.~Schwartz in \cite{Schwartz}. There, given a quasi complete 
(meaning that any closed bounded subset is complete) locally convex space $\cF$, 
the object of investigation is provided by Hilbert spaces continuously included in $\cF$.
\end{remark}

\begin{proposition}\label{p:unic}
Let $\cH$ be a reproducing kernel space on $X$ with reproducing kernel $K$. Then $\cH$ is uniquely 
determined by $K$.
\end{proposition}

\begin{proof} Let $\cG\subseteq \cF_\FF(X)$ be another reproducing kernel Hilbert space on $X$ with the same 
kernel $K$. By Proposition~\ref{p:rkhs}, the set $\{K_x\mid x\in X\}$ is total, that is, the vector space 
$\Span\{K_x\mid x\in X\}$ is dense, in both $\cH$ and $\cG$. In addition, for any $n\in\NN$, $x_1,\ldots,x_n\in X$, and $\alpha_1,\ldots,\alpha_n\in\FF$, letting 
\begin{equation*}
f=\sum_{j=1}^n \alpha_jK_{x_j}\in \Span\{K_x\mid x\in X\}\subseteq \cH\cap \cG,
\end{equation*}
we have
\begin{align*}
\|f\|_\cH^2 & = \langle f,f\rangle_\cH= \sum_{i,j=1}^n \alpha_i\ol{\alpha_j} \langle K_{x_i},K_{x_j}\rangle_\cH \\
& = \sum_{i,j=1}^n \alpha_i\ol{\alpha_j} K(x_j,x_i)  = \sum_{i,j=1}^n \alpha_i\ol{\alpha_j} \langle K_{x_i},K_{x_j}\rangle_\cG \\
& = \langle f,f\rangle_\cG=\|f\|_\cG^2,
\end{align*}
hence, on $\Span\{K_x\mid x\in X\}$ the two norms $\|\cdot\|_\cH$ and $\|\cdot\|_\cG$ coincide.

Let $f\in\cH$ be arbitrary. Then there exists a sequence $(f_n)_n$ of functions $f_n\in \Span\{K_x\mid x\in X\}$ 
such that $f_n\ra f$ as $n\ra\infty$ with respect to the norm $\|\cdot\|_\cH$. Since the two norms coincide on 
$\Span\{K_x\mid x\in X\}$, the sequence $(f_n)_n$ is Cauchy with respect to the norm $\|\cdot\|_\cG$ as 
well and hence there exists $g\in\cG$ such that $f_n\ra g$ as $n\ra\infty$ with respect to the norm $\|\cdot\|_\cG$.
Since convergence in $\cH$ and $\cG$ imply pointwise convergence, it follows that $f_n(x)\ra f(x)$ and 
$f_n(x)\ra g(x)$ as $n\ra\infty$ for all $x\in X$, hence $f=g\in\cG$. This shows that $\cH\subseteq \cG$. By 
symmetry we have $\cG\subseteq \cH$ as well, hence $\cH=\cG$.
\end{proof}

\begin{remark}
In general, reproducing kernel Hilbert spaces may have any dimension, in particular they may be nonseparable.
Let $X$ be an arbitrary nonempty set and define
\begin{equation*}
\cH=\ell^2_\FF(X):=\{f\colon X\ra\FF\mid \sum_{x\in X}|f(x)|^2<\infty\},
\end{equation*}
with inner product and norm
\begin{equation*}
\langle f,g\rangle_\cH:=\sum_{x\in X} f(x)\ol{g(x)},\quad \|f\|_\cH^2:=\sum_{x\in X}|f(x)|^2,\quad f,g\in \cH.
\end{equation*}
Then $\{\delta_x\mid x\in X\}$ is an orthonormal basis of $\cH$ hence $\dim_{\mathrm{H}}(\cH)=\card(X)$. 

$\cH$ is a reproducing kernel Hilbert space since for each $x\in X$ the evaluation functional $h\mapsto h(x)$ is 
continuous.  Actually, the reproducing kernel of $\cH$ is the identity kernel $K$,
\begin{equation*}
K(x,y)=\begin{cases} 1, & x=y, \\ 0, & x\neq y,
\end{cases}
\end{equation*}
equivalently, $K_x=\delta_x$ for all $x\in X$.
\end{remark}

The next theorem says that the reproducing kernel $K$ 
can be recovered explicitly from any orthonormal basis of its reproducing kernel Hilbert space $\cH_K$. In the 
proof of Mercer's Theorem, this fact, together with the fact that the convergence of the series is pointwise, 
will play an important role. In the second part of this theorem we can get another characterisation of those Hilbert 
spaces of functions on a given set $X$ which are reproducing kernels, in terms of a Bessel type condition 
of orthonormal bases.

\begin{theorem}\label{t:on}   Let $\cH$ be a Hilbert space contained in $\cF_\FF(X)$, for some set $X$, with all 
algebraic operations. 

\nr{1} If $\cH$ is a reproducing kernel Hilbert space with reproducing kernel $K$ then, for any orthonormal basis 
$\{e_j\mid j\in\cJ\}$ of $\cH$ and for all $x,y\in X$, we have
\begin{equation*}
K(x,y)=\sum_{j\in\cJ} e_j(x)\ol{e_j(y)},
\end{equation*}
where the convergence is pointwise in the sense of summability. 

\nr{2} If there exists an orthonormal basis $\{e_j\}_{j\in\cJ}$ of $\cH$ such that, for all $x\in X$ we have
\begin{equation}\label{e:suje}
\sum_{j\in\cJ} |e_j(x)|^2<\infty,
\end{equation}
then $\cH$ is a reproducing kernel Hilbert space on $X$. More precisely, for each $x\in X$ there exists uniquely a vector $K_x\in\cH$ such that
\begin{equation*}
K_x:=\sum_{j\in \cJ} \ol{e_j(x)}e_j,\end{equation*}
where the series converges in $\cH$, and letting $K(x,y)=\langle K_y,K_x\rangle_\cH$, it follows that $K$ is the reproducing 
kernel of $\cH$.
\end{theorem}

\begin{proof} (1) Let $x,y\in X$ be arbitrary. By considering the Fourier expansions of the functions $K_x$ and $K_y$, which are in $\cH$, with respect to the orthonormal basis $\{e_j\mid j\in \cJ\}$ we have
\begin{equation*}
K_x=\sum_{j\in\cJ} \langle K_x,e_j\rangle_\cH e_j,\quad K_y=\sum_{j\in\cJ} \langle K_y,e_j\rangle_\cH e_j,
\end{equation*}
where the convergence holds with respect to the norm $\|\cdot\|_\cH$ in the sense of summability. Actually, by 
the reproducing property we know that the Fourier coefficients are
\begin{equation}\label{e:olex}
\ol{e_j(x)}=\langle K_x,e_j\rangle_\cH,\quad \ol{e_j(y)}=\langle K_y,e_j\rangle_\cH,
\end{equation}
and that they are nontrivial only for countably many $j\in \cJ$ hence, in particular, the Fourier series are absolutely converging series in the classical sense. Then
\begin{align*}
K(x,y) & = \langle K_y,K_x\rangle_\cH =\bigl\langle \sum_{j\in\cJ} \langle K_y,e_j\rangle_\cH e_j, \sum_{l\in\cJ} \langle K_x,e_l\rangle_\cH e_l\bigr\rangle_\cH \\
\intertext{hence, since convergence holds in $\cH$, in view of \eqref{e:olex}, this equals}
& = \sum_{j,l\in\cJ} \ol{e_j(y)}e_l(x)\langle e_j,e_l\rangle_\cH  = \sum_{j\in\cJ}e_j(x) \ol{e_j(y)},
\end{align*}
the series being pointwise  convergent.

(2) Let $y\in X$ be arbitrary but fixed. By \eqref{e:suje} for $y$, it follows that we have a Fourier series 
\begin{equation}\label{e:kays}
K_y:=\sum_{j\in \cJ} \ol{e_j(y)}e_j
\end{equation}
that converges in the strong topology of $\cH$. Now, if $f\in \cH$ is arbitrary, using the Fourier series of $f$, we have
\begin{equation}\label{e:fesumej}
f=\sum_{j\in\cJ} \langle f,e_j\rangle_\cH e_j,
\end{equation}
where the sum converges in the strong topology of $\cH$ and the Parseval Identity holds, that is,
\begin{equation*}
\|f\|_\cH^2=\sum_{j\in\cJ} |\langle f,e_j\rangle_\cH|^2.
\end{equation*} 
For arbitrary $x\in X$ we observe that, by the Cauchy Inequality, the Parseval Identity, and \eqref{e:suje},
we have
\begin{align*}
\sum_{j\in\cJ} |\langle f,e_j\rangle_\cH e_j(x)| & = \sum_{j\in\cJ} |\langle f,e_j\rangle_\cH|  |e_j(x)| \\
& \leq
\bigl(\sum_{j\in\cJ} |\langle f,e_j\rangle_\cH|^2\bigr)^{1/2 }\bigl(\sum_{j\in\cJ} |e_j(x)|^2\bigr)^{1/2} = C_x \|f\|_\cH,
\end{align*}
hence the sum $\sum_{j\in\cJ} \langle f,e_j\rangle_\cH e_j(x)$ converges absolutely and then we get that the series in \eqref{e:fesumej} converges pointwise as well and
\begin{equation*}
f(x)=\sum_{j\in\cJ} \langle f,e_j\rangle_\cH e_j(x).
\end{equation*}
Then,
\begin{align*}
\langle f,K_x\rangle_\cH & = \bigl\langle \sum_{i\in \cJ}\langle f,e_i\rangle_\cH e_i,\sum_{j\in\cJ} \ol{e_j(x)}
e_j\bigl\rangle_\cH \\ 
& = \sum_{i\in\cJ}\sum_{j\in\cJ} \langle f,e_i\rangle_\cH e_j(x)\langle e_i,e_j\rangle_\cH \\
& = \sum_{j\in\cJ}\langle f,e_j\rangle_\cH e_j(x) =f(x).
\end{align*}
Letting $K(x,y):=\langle K_y,K_x\rangle_\cH$, for all $x,y\in X$, it follows that $\cH$ 
is the reproducing kernel Hilbert space associated to $K$.
\end{proof}

\begin{remark} Given a Hilbert space $\cH$, a family of vectors $\{f_j\mid j\in\cJ\}$ in $\cH$ is called 
a \emph{Parseval frame} if for any $h\in\cH$ the Parseval Identity holds
\begin{equation*}
\|h\|_\cH^2=\sum_{j\in\cJ} |\langle h,f_j\rangle_\cH|^2.
\end{equation*}
Clearly any orthonormal basis is a Parseval frame. However,
the concept of a Parseval frame is more general than that of an orthonormal basis. For example, if 
$\{e_j\mid j\in\cJ\}$ and $\{g_j\mid j\in \cJ\}$ are two orthonormal bases of $\cH$ then the family 
$\{\frac{1}{\sqrt{2}} e_j,\frac{1}{\sqrt{2}}g_j\mid j\in\cJ\}$ is a Parseval frame.

Theorem~\ref{t:on} is true for Parseval frames as well, but the proof is more involved, see 
 \cite{PaulsenRaghupathi}.
\end{remark}

The main theorem of this section is the next one saying that to any Hermitian/symmeric positive semidefinite 
kernel one can associate a (unique, by Proposition~\ref{p:unic}) reproducing kernel Hilbert space. 
Here, by means of  Theorem \ref{t:kolmo}, we take an important 
shortcut with respect to the classical proofs.

\begin{theorem}[E.H.~Moore \cite{Moore}]\label{t:moore} 
For any Hermitian/symmetric positive semidefinite kernel 
$K\colon X\times X\ra\FF$ there exists a unique reproducing kernel Hilbert space $\cH_K$ such that $K$ is its 
reproducing kernel.
\end{theorem}

\begin{proof} By Theorem \ref{t:kolmo}, 
there exists a minimal linearisation $(\cG,\Phi)$ of $K$, that is, 
$\cG$ is a Hilbert space, $\Phi\colon X\ra\cG$, $\Span\Phi(X)$ is dense in $\cG$, and 
$K(x,y)=\langle \Phi(y),\Phi(x)\rangle_\cG$ for all $x,y\in X$. Consider the map $U\colon \cG\ra\cF_\FF(X)$ 
defined by
\begin{equation}\label{e:ug}
\cG\ni g\mapsto Ug:= \langle g,\Phi(\cdot)\rangle_\cG\in\cF_\FF(X).
\end{equation}
This map is clearly linear and injective, because, if for $g\in\cG$ we have $\langle g,\Phi(x)\rangle_\cG=0$ for all 
$X$, then $g\perp \Span\Phi(X)$, which is dense in $\cG$, hence $g=0$.

Let $\cH:=\ran(U)$ and on it define the inner product $\langle\cdot,\cdot\rangle_\cH$ in the following way.
For any $h,g\in\cG$,
\begin{equation}\label{e:uhaug} \langle Uh,Ug\rangle_\cH=
\langle\langle h,\Phi(\cdot)\rangle_\cG,\langle g,\Phi(\cdot)\rangle_\cG\rangle_\cH:=\langle h,g\rangle_\cG.
\end{equation}
In this fashion, $\cH$ becomes a Hilbert space in $\cF_\FF(X)$ in such a way that $U\colon \cG\ra \cH$
is a unitary operator. We claim that $\cH$ is the reproducing kernel Hilbert space on $X$ with kernel $K$.

Indeed, $\cH$ is a Hilbert space in $\cF_\FF(X)$ and, for each $x\in X$, we have
\begin{equation*}
K_x(y):=K(y,x)=\langle \Phi(x),\Phi(y)\rangle_\cG,\quad y\in X,
\end{equation*}
hence $K_x=\langle \Phi(x),\Phi(\cdot)\rangle_\cG\in \ran(U)=\cH$, since $\Phi(x)\in\cG$.

Finally, for arbitrary $f\in\cH$ there exists uniquely $g\in\cG$ such that $f=\langle g,\Phi(\cdot)\rangle_\cG$. Then, 
for any $x\in X$ we have
\begin{equation*} \langle f,K_x\rangle_\cH = \langle \langle g,\Phi(\cdot)\rangle_\cG, \langle \Phi(x),\Phi(\cdot)\rangle_\cG\rangle_\cH=\langle g,\Phi(x)\rangle_\cG=f(x),
\end{equation*}
hence the reproducing property is proven and the claim as well.
\end{proof}

\begin{remark} The classical proof of Moore's Theorem goes differently. Given a Hermitian/symmetric 
and positive semidefinite kernel $K$ on $X$, in order to construct the reproducing kernel space $\cH_K$ we have the vector space $\Span\{K_x\mid x\in X\}$ on which we have to build it. The original construction was performed in three steps.\medskip

\textbf{Step 1.} On $\cS:=\Span\{K_x\mid x\in X\}$ define
\begin{equation}\label{e:lasum}
\langle \sum_{i=1}^m \alpha_iK_{x_i},\sum_{j=1}^n \beta_j K_{y_j}\rangle_K:= \sum_{i=1}^m\sum_{j=1}^n
\alpha_i\ol{\beta_j}K(y_j,x_i)
\end{equation}
and show that it is correctly defined.\medskip

\textbf{Step 2.} Show that $\langle\cdot,\cdot\rangle_K$ is an inner product.\medskip

\textbf{Step 3.} Show that the inner product space $(\cS,\langle\cdot,\cdot\rangle_K)$ can be completed to a 
Hilbert space within $\cF_\FF(X)$.\medskip

The fact that we use Kolmogorov's Theorem makes an important shortcut through all these.
\end{remark}

\begin{remark}\label{r:schwartz} If $X$ is a nonempty set, consider the following sets.
\begin{equation*}
\cK_\FF(X):=\{K\colon X\times X\ra \FF\mid K \mbox{ is a kernel}\},
\end{equation*}
\begin{equation*}
\cK_\FF^{\mathrm{H}}(X):=\{K\in \cK_\FF(X)\mid K=K^*\},
\end{equation*}
where $K^*(x,y)=\ol{K(y,x)}$ for all $x,y\in X$, and
\begin{equation*}
\cK_\FF^+(X):=\{K\in \cK_\FF^{\mathrm{H}}(X)\mid K\mbox{ is positive semidefinite}\}.
\end{equation*}
Note that: $\cK_\FF(X)$ is a vector space over $\FF$, $\cK_\FF^{\mathrm{H}}$ is a real vector space, while
$\cK_\FF^+(X)$ is a convex cone, closed with respect to the topology of pointwise convergence.

On $\cK_\FF^+(X)$ an order relation can be defined: for $K,L\in\cF_\FF^+(X)$, we say that $K\leq L$ if $L-K$ is 
positive semidefinite, which we express by $L-K\geq 0$. Then this order relation can be extended to the whole 
space $\cK_\FF^{\mathrm{H}}(X)$: 
given $K,L \in \cK_\FF^{\mathrm{H}}(X)$ we say that $K\leq L$ if $L-K\geq 0$.
\end{remark}

In the following we need two technical results from operator theory. These results are considered to be folklore
by the community of operator theorists but they have been recorded and generalised to the case of unbounded 
operators by R.G.~Douglas in \cite{Douglas}.

\begin{lemma}\label{l:douglas} Let $\cE$, $\cG$, and $\cH$ be Hilbert spaces over the same field $\FF$, and let 
$T\in\cB(\cE,\cH)$ and $S\in\cB(\cG,\cH)$. The following assertions are equivalent.
\begin{itemize}
\item[(i)] $\ran(T)\subseteq\ran(S)$. 
\item[(ii)] There exists $\alpha>0$ such that $TT^*\leq \alpha SS^*$.
\item[(iii)] There exists $F\in \cB(\cE,\cG)$ such that $T=SF$.
\end{itemize}
In addition, the operator $F$ as in (iii), if exists, can be always chosen such that 
$\ran(F)\subseteq \cG\ominus \ker(S)$ and, in this case, it is unique with these properties.
\end{lemma}

\begin{proof} (i)$\Ra$((iii). Assuming that $\ran(T)\subseteq \ran(S)$, 
let $S^{-1}\colon \ran(S)\ra \cG\ominus \ker(S)$ denote 
the inverse on the range of the operator $S\colon \cG\ominus\ker(S)\ra \ran(S)$. Since $S$ is bounded, the 
operator $S^{-1}$ is closed. Define $F\colon \cE\ra \cG\ominus\ker(S)$ by $F=S^{-1}T$. 
Since $\ran(T)\subseteq \ran(S)$, $F$ is correctly defined and, since $T$ is bounded and $S^{-1}$ is closed, it 
is a closed operator. By the Closed Graph Theorem, it follows that $F$ is bounded. We lift $F\colon \cE\ra \cG$ 
by composing it with the isometric inclusion $\cG\ominus\ker(S)\hookrightarrow \cG$ and note that $T=SF$.
Also, note that $\ran(F)\subseteq \cG\ominus\ker(S)$ and that $F$ is unique with these properties.

(iii)$\Ra$(i). Clear.

(iii)$\Ra$(ii). If $T=SF$ for some $F\in\cB(\cE,\cG)$ then
\begin{equation*}
TT^*=SFF^*S^*\leq \|F\|^2 SS^*,
\end{equation*}
and let $\alpha\geq \|F\|^2$ be a nontrivial number.

(ii)$\Ra$(iii). Assume that $TT^*\leq \alpha SS^*$. Then, for any $h\in \cH$ we have
\begin{equation*}
\|T^*h\|^2_\cE=\langle TT^*h,h\rangle_\cH\leq \alpha \langle SS^*h,h\rangle_\cH=\alpha \|S^*h\|^2_\cG,
\end{equation*}
which shows that letting $E\colon \ran(S^*)\ra \ran(T^*)$ be defined by $S^*h\mapsto T^*h$, for $h\in \cH$ 
arbitrary, this is a correctly defined linear operator and $\|E\|\leq\sqrt{\alpha}$. Then, extend $E$ to the closure of
$\ran(S^*)$ and then trivially to the whole space $\cG$ and observe that $ES^*=T^*$. Letting $F=E^*$ we 
have $T=SF$. 
\end{proof}

Recall that, a bounded linear operator $T\in \cB(\cE,\cH)$, for some Hilbert spaces $\cE$ and $\cH$ 
over the same field $\FF$, is a \emph{partial isometry} if $T^*T$ and $TT^*$ are orthogonal projections, 
equivalently, $T^*TT^*=T^*$ and $TT^*T=T$. In this case, the space $\ran(T^*T)=\cE\ominus\ker(T)$ is called
the \emph{right support} of $T$ and the space $\ran(TT^*)=\cH\ominus \ker(T^*)$ is called the 
\emph{left support} of $T$.

\begin{lemma}\label{l:pai} Let $\cE$, $\cG$, and $\cH$ be Hilbert spaces over the same field $\FF$, and let 
$T\in\cB(\cE,\cH)$ and $S\in\cB(\cG,\cH)$ be such that $TT^*=SS^*$. Then $\ran(T)=\ran(S)$ and 
$T=SV$ for $V\in\cB(\cE,\cG)$ a 
unique partial isometry, with  the right support of $V$ equal to the right support of $T$ and the left support of $V$ 
equal to the right support of $S$. 
\end{lemma}

\begin{proof} If $TT^*=SS^*$ then $\ran(T)=\ran(S)$ from Lemma~\ref{l:douglas}. Moreover, for any $h\in\cH$ we have
\begin{equation*} \|T^*h\|^2_\cE=\langle TT^*h,h\rangle_\cH=\langle SS^*h,h\rangle_\cH=\|S^*h\|_\cG,
\end{equation*}
hence, the linear 
operator $U\colon \ran(S^*)\ra \ran(T^*)$ is correctly defined by $S^*h\mapsto T^*h$, for $h\in\cH$,
and isometric. Then, we can uniquely extend $U$ to a partial isometry $U\colon \cG\ra\cE$
such that the left support of $U$ equals the right support of $T$ and the right support of $U$ is the right 
support of $S$. Letting $V=U^*$, we have $T=SV$.
\end{proof}

\begin{remark}\label{r:opran} 
The previous two lemmas are related to the theory of operator ranges in a Hilbert space. Basically, 
this theory says that the 
subspaces of a given Hilbert space $\cH$ that are ranges of operators (bounded or closed, is the 
same) are exactly those Hilbert spaces that are continuously included in $\cH$ and
share many properties with closed subspaces but, in the same time, they have new properties. 
Briefly, let $\cH$ be a Hilbert space and $\cG$ 
another Hilbert space such that $\cG\subseteq \cH$, with all algebraic operations, and let us assume that the 
inclusion operator $J\colon \cG\hookrightarrow\cH$ is continuous. Let $A:=JJ^*\in\cB(\cH)$ and note that 
$A\geq 0$, hence its square root $A^{1/2}\in\cB(\cH)$ exists. Then, from Lemma~\ref{l:douglas} it follows that
$\cG=\ran(J)=\ran(A^{1/2})$ and that $A^{1/2}=JV$ with $V\colon \cH\ominus\ker(A)\ra \cG$ a unitary operator. 
Since $J$ acts like the identity operator it follows that, when viewing $A^{1/2}\colon \cH\ominus\ker(A)\ra \cG$, it 
is a unitary operator. In particular, 
\begin{equation*}
\langle A^{1/2}h,A^{1/2}g\rangle_\cG=\langle h,g\rangle_\cH,\quad h,g\in\cH\ominus\ker(A).
\end{equation*}

Conversely, let $T\in\cB(\cK,\cH)$ and consider $\cG:=\ran(T)$ endowed with the inner product
\begin{equation*}
\langle Th,Tg\rangle_{\cG}:=\langle h,g\rangle_\cK,\quad h,g\in\cK\ominus\ker(T).
\end{equation*}
Then, for arbitrary $g\in\cG$ there exists unique $k\in \cK\ominus\ker(T)$ such that $g=Tk$ and then
\begin{equation*}
\|g\|_\cH=\|Tk\|_\cH\leq \|T\| \|k\|_\cK=\|T\| \|Tk\|_\cG=\|T\| \|g\|_\cG.
\end{equation*}
This proves that the canonical embedding $J\colon \cG\hookrightarrow \cH$ is continuous.

In Theorem~\ref{t:inclusion} we see that the concept of a reproducing kernel Hilbert space is actually
originating from the concept of operator range, sharing a property called automatic continuity: once an inclusion 
of operator ranges, or reproducing kernel Hilbert spaces, holds, the inclusion is continuous. 
A survey of this theory can be found at P.A.~Fillmore and 
J.P.~Williams \cite{FillmoreWilliams}. This theory is closely related to the theory of reproducing kernel 
Hilbert space, and this is the best viewed through the approach of L.~Schwartz \cite{Schwartz}. We will better 
see this in Theorem~\ref{t:leka} when it will be proven that the reproducing kernel Hilbert space associated to 
a Mercer kernel is an operator range.
\end{remark}

\begin{theorem}[Aronszajn's Inclusion Theorem \cite{Aronszajn1}, \cite{Aronszajn2}]
\label{t:inclusion} Given $K,L\in\cK_\FF^+(X)$, the following assertions are equivalent.
\begin{itemize}
\item[(i)] $\cH_K\subseteq \cH_L$.
\item[(ii)] $\cH_K\hookrightarrow \cH_L$, that is, $\cH_K$ is continuously included in $\cH_L$.
\item[(iii)] There exists $c>0$ such that $K\leq c L$.
\end{itemize}
\end{theorem}

\begin{proof} (i)$\Ra$(ii). If $\cH_K\subseteq \cH_L$ let $J_{L,K}\colon\cH_K\ra\cH_L$ be the inclusion operator: 
$J_{L,K}h=h$ for all $h\in\cH_K$. We show that this operator is closed. Indeed, if $(h_n)_n$ is a sequence in 
$\cH_K$ converging to $h$ with respect to the norm $\|\cdot\|_{\cH_K}$  and to $g$ with respect to the norm 
$\|\cdot\|_{\cH_L}$ then it converges pointwise to $h$ and $g$ hence $h=g$. By the Closed Graph Theorem it 
follows that $J_{L,K}$ is bounded.

(ii)$\Ra$(iii). 
If $\cH_K$ is continuously included in $\cH_L$ let $J_{L,K}\colon \cH_K\hookrightarrow \cH_L$ and let 
$\|J_{L,K}\|<\infty$ denote its operator norm. Let $x_1,\ldots,x_n\in X$ and $\alpha_1,\ldots,\alpha_n\in\FF$ be 
arbitrary. Then,
\begin{align*}
\bigl\| \sum_{j=1}^n \alpha_j K_{x_j}\|^2_{\cH_K} & = \sum_{i,j=1}^n \ol{\alpha_i}\alpha_j K_{x_j}(x_i) \\
\intertext{which, using the fact that $K_{x_j}\in \cH_L$ and the reproducing property in $\cH_L$, equals}
& = \sum_{i,j=1}^n \ol{\alpha_i}\alpha_j \langle K_{x_j},L_{x_i}\rangle_{\cH_L} \\
& = \bigl\langle \sum_{j=1}^n \alpha_j K_{x_j},\sum_{i=1}^n \alpha_i L_{x_i}\bigr\rangle_{\cH_L} \\
& \leq \bigl\| \sum_{j=1}^n \alpha_j K_{x_j}\bigr\|_{\cH_L} \bigl\| \sum_{i=1}^n \alpha_iL_{x_i}\bigr\|_{\cH_L} \\
\intertext{which, taking into account that $J_{L,K}$ is a bounded operator, is dominated by}
& \leq \|J_{L,K}\| \bigl\| \sum_{j=1}^n \alpha_j K_{x_j}\bigr\|_{\cH_K} \bigl\| \sum_{i=1}^n \alpha_iL_{x_i}\bigr\|_{\cH_L}.
\end{align*}
From here, a standard argument shows that
\begin{equation*}
\sum_{i,j=1}^n \ol{\alpha_i}\alpha_j K(x_i,x_j) \leq \|J_{L,K}\|^2 \sum_{i,j=1}^n \ol{\alpha_i}\alpha_j L(x_i,x_j). 
\end{equation*}
We have proven that $K\leq \|J_{L,K}\|^2 L$.

(iii)$\Ra$(i). This basically follows by the constructions in Theorem~\ref{t:kolmo}, Theorem~\ref{t:moore}, and the 
uniqueness of the reproducing kernel Hilbert space induced by a kernel, by Proposition~\ref{p:unic}. 
With notation as in \eqref{e:lasum}, 
letting $V\colon \Span\{L_x\mid x\in X\}\ra \Span\{K_x\mid x\in X\}$ be defined by
\begin{equation}\label{e:vebi} V\bigl(\sum_{j=1}^n \alpha_j L_{x_j}\bigr) := \sum_{j=1}^n \alpha_j K_{x_j},
\end{equation}
for arbitrary $n\in\NN$, $x_1,\ldots,x_n\in X$, and $\alpha_1,\ldots, \alpha_j\in\FF$, we have
\begin{align*}
\bigl\|V\bigl(\sum_{j=1}^n \alpha_j L_{x_j}\bigr)\bigr\|_K^2 & = \bigl\langle V\bigl(\sum_{j=1}^n \alpha_j L_{x_j} 
\bigr),V\bigr(\sum_{i=1}^n \alpha_i L_{x_i}\bigr)\bigr\rangle_K  
= \langle \sum_{j=1}^n \alpha_j K_{x_j},\sum_{i=1}^n \alpha_i K_{x_i}\rangle_K \\
 & = \sum_{i,j=1}^n \ol{\alpha_i}\alpha_j K(x_i,x_j)  \leq c  \sum_{i,j=1}^n \ol{\alpha_i}\alpha_j L(x_i,x_j)\\
 & = c  \langle \sum_{j=1}^n \alpha_j L_{x_j},\sum_{i=1}^n \alpha_i L_{x_i}\rangle_L 
 = c \bigl\| \sum_{j=1}^n \alpha_j L_{x_j}\bigr\|^2_L, 
\end{align*}
which shows that $V$ is correctly defined and it extends uniquely to a bounded linear operator 
$V\colon\cH_L\ra\cH_K$. As usual, let $|V|=(V^*V)^{1/2}\in\cB^+(\cH_L)$.

Since, for any $x,y\in X$ we have
\begin{align*}
K(y,x)=\langle K_x,K_y\rangle_{\cH_K} & = \langle V L_x,VL_y\rangle_{\cH_K} 
= \langle V^*V L_x,L_y\rangle_{\cH_L} =
\langle |V| L_x,|V|L_y\rangle_{\cH_L},
\end{align*}
it follows that $(\cH_l\ominus\ker(V);\{|V|L_x\}_{x\in X})$ is a minimal linearsiation of $K$ and then, 
by the proof of
Theorem~\ref{t:moore}, see \eqref{e:ug} and \eqref{e:uhaug}, the reproducing kernel Hilbert space $\cH_K$
consists in all functions $X\ni x\mapsto \langle h,|V|L_x\rangle_{\cH_L}\in \CC$, for a unique 
$h\in \cH_L\ominus\ker(V)$. But, 
\begin{equation}\label{e:lahiv}
\langle h,|V|L_x\rangle_{\cH_L}=\langle |V|h,L_x\rangle_{\cH_L},\quad h\in \cH_L\ominus\ker(V),\ x\in X,
\end{equation}
and, again by Theorem~\ref{t:moore} and Proposition~\ref{p:unic}, it follows that all these functions belong to 
$\cH_L$ as well, hence $\cH_K\subseteq \cH_L$. 
\end{proof}

We point out a statement that we obtained during the proof of the implication (iii)$\Ra$(i), that we have just 
proven. This strengthens one more time the statement that reproducing kernel Hilbert spaces are essentially 
operator ranges. 

\begin{corollary}\label{c:rec} Given $K,L\in\cK_\FF^+(X)$, assume that
there exists $c>0$ such that $K\leq c L$. Let $\cR:=\ran(|V|)\subseteq \cH_L$, where $V\in\cB(\cH_L,\cH_K)$ is 
defined at \eqref{e:vebi} and $|V|=(V^*V)^{1/2}$, 
be the Hilbert space with inner product
\begin{equation*}
\langle |V|f,|V|g\rangle_\cR:=\langle f,g\rangle_{\cH_L},
\end{equation*}
where $f,g\in \cH_L\ominus \ker(V)$ are unique. Then
$\cH_K=\cR$ as reproducing kernel Hilbert spaces.
\end{corollary}

\begin{proof} As in the proof of (iii)$\Ra$(i) before, $V$ is correctly defined and it extends uniquely to a bounded 
linear operator $V\colon\cH_L\ra\cH_K$.
We first observe that, when viewed 
$|V|\colon \cH_L\ominus \ker(V)\ra \cR$, it is a unitary operator. In view of Remark~\ref{r:opran}, $\cR$ is 
continuously included in $\cH_L$, in particular its topology is stronger than the topology of pointwise 
convergence hence, in view of Proposition~\ref{p:conte}, it is a reproducing kernel Hilbert space on $X$. 
By the proof of Theorem~\ref{t:moore} and 
Proposition~\ref{p:unic}, from \eqref{e:lahiv} it follows that
$\cH_K=\cR$ as reproducing kernel Hilbert spaces.
\end{proof}

The previous theorem can be used in order to provide an answer to the following question: 
\emph{given a kernel $K\in\cK_\FF^+(X)$ and a nontrivial 
function $\phi\in \cF_\FF(X)$, how can we know that it belongs to the space $\cH_K$?} In order to provide an 
answer to this question, we need some notation. For each function $\phi\in\cF_\FF(X)$ consider the kernel
\begin{equation*}
K_\phi(x,y):=\phi(x)\ol{\phi(y)},\quad x,y\in X,
\end{equation*}
and observe that $K_\phi\in\cK_\FF^+(X)$. Also,
\begin{equation*}
(C_{K_\phi}f)(x)=\sum_{y\in X} \phi(x)\ol{\phi(y)}f(y)=\phi(x)\sum_{y\in X} \ol{\phi(y)}f(y) ,\quad f\in\cF_\FF^0(X),
\end{equation*}
hence 
$\ran(C_{K_\phi})$ has dimension $1$ and is generated by $\phi$. In particular, $\cH_{K_\phi}=\FF\phi$.

\begin{corollary}\label{c:in} 
Given $K\in\cK_\FF^+(X)$ and $\phi\in\cF_\FF(X)$, the following assertions are equivalent.
\begin{itemize}
\item[(i)] $\phi\in\cH_K$.
\item[(ii)] $K_\phi\leq c K$ for some $c>0$.
\item[(iii)] $\cH_{K_\phi}\subseteq \cH_K$.
\end{itemize}
\end{corollary}

\begin{example} Let $n\in\NN$ be arbitrary and denote $\NN_n:=\{1,2,\ldots,n\}$. Consider a kernel 
$K\in\cK_\FF(\NN_n)$, that is, an $n\times n$ matrix $K$ with entries $K(i,j)$, $i,j=1,\ldots,n$. Let 
$K_j:=K(\cdot,j)$, that is, $K_j$ is the $j$-th column of the matrix $K$. Then, letting $\FF^n_2$ 
denote the vector space $\FF^n$ organised as a Hilbert space in the usual manner,
\begin{equation*}
\cF_\FF(\NN_n)=\cF_\FF^0(\NN_n)=\FF_2^n,
\end{equation*}
the operator $C_K\colon \cF_\FF(\NN_n)\ra\cF_\FF(\NN_n)$, defined by
\begin{equation*}
(C_K f)(i):= \sum_{j=1}^n K(i,j)f(j)=\sum_{j=1}^n K_j(i)f(j),\quad i=1,\ldots,n,
\end{equation*}
$C_K$ is nothing else than the operator of multiplication with matrix $K$ acting $\FF^n_2\ra \FF^n_2$.
Then,
$\ran(C_K)=\Span\{K_j\mid j=1,\ldots,n\}$ is the range space of the matrix $K$, and on it we have the inner 
product $\langle \cdot,\cdot\rangle_K$ defined by
\begin{equation*}
\langle \sum_{j=1}^n \alpha_jK_j,\sum_{l=1}^n\beta_lK_l\rangle_K:=\sum_{j,l=1}^n K(l,j)\alpha_j\ol{\beta_l},
\end{equation*}
hence $\cH_K=\ran(C_K)=\ran(K)$ with inner product $\langle\cdot,\cdot\rangle_K$ and 
$\dim(\cH_K)=\rank(K)$.

Let $J_K\colon \cH_K\hookrightarrow \FF^n_2$. Then,
\begin{equation*}
\langle J_Kf,g\rangle_{\FF^n_2} =\langle f,C_Kg\rangle_K,\quad f\in\cH_K,g\in \FF^n_2,
\end{equation*}
which means that $C_K=J_K^*\colon \FF^n_2\ra \cH_K$. But, considering $C_K\colon\FF^n_2\ra\FF_2^n$, we actually have $C_K=J_KJ_K^*$.
In fact, we have that the operator
\begin{equation*}
C_K^{1/2}\colon \FF^n_2\ominus \ker(C_K)\ra \cH_K
\end{equation*}
 is an isometric isomorphism of Hilbert spaces. These facts are consequences of the considerations from 
 Remark~\ref{r:opran}.  In Theorem~\ref{t:leka}
 more general results will be obtained.
\end{example}

\section{Weyl's and the Gaussian Kernels}\label{s:wgk}

In this section we present two of the most useful kernels in machine learning, Weyl's 
kernel and the Gaussian 
kernel. Details and much more properties 
of the latter can be found at H.Q.~Minh \cite{Minh}, which we followed in writing this section.

\subsection{Weyl's Kernel}\label{ss:wk}

We start with some notation. For $n\in\NN$ and a multiindex 
$\alpha=(\alpha_1,\alpha_2,\ldots,\alpha_n)\in\NN_0^n$,
$n$ is the \emph{dimension} of $\alpha$ and
$|\alpha|:=\alpha_1+\alpha_2+\cdots+ \alpha_n$ is the \emph{length} of $\alpha$. Also,
$\alpha!:= \alpha_1!\cdot \alpha_2!\cdots\alpha_n!$ is the \emph{factorial} of $\alpha$.
If $x\in\RR^n$ then $x^\alpha:=x_1^{\alpha_1}\cdot x_2^{\alpha2}\cdots \alpha_n^{\alpha_n}$.
If $x,y\in\RR^n$ then $\langle x,y\rangle=\sum_{j=1}^n x_j y_j$ is the Euclidean inner product and 
$\|x\|=\sqrt{\langle x,x\rangle}$ is the Euclidean norm.

For $d\in\NN_0$ we denote by
\begin{equation*}C^\alpha_d={d \choose \alpha}={d \choose \alpha_1,\ldots,\alpha_n}:=\frac{d!}{\alpha!}=
\frac{d!}{\alpha_1!\cdots \alpha_n!},
\end{equation*}
the \emph{multinomial coefficients}
that show up in the Multinomial Theorem,
\begin{equation*}
(x_1+x_2+\cdots+x_n)^d =\sum_{\alpha\in\NN_0^n(d)} C^\alpha_d x^\alpha,
\end{equation*}
where
\begin{equation*}
\NN_0^n(d):=\{\alpha\in\NN_0^n \mid |\alpha|=d\}.
\end{equation*}

For $d\in\NN$ Weyl's Kernel $W^{(d)}\colon \RR^n \times \RR^n\ra \RR$ is defined as
\begin{equation}\label{e:wedexy}
W^{(d)}(x,y):=\langle x,y\rangle^d=\bigl(\sum_{j=1}^n x_j y_j\bigr)^d
=\sum_{\alpha\in\NN_0^n(d)} C^\alpha_d x^\alpha y^\alpha,\quad x,y\in\RR^n.
\end{equation} 

\begin{figure}[h]
\includegraphics[width=14cm, trim = 0cm 8cm 0cm 8cm]{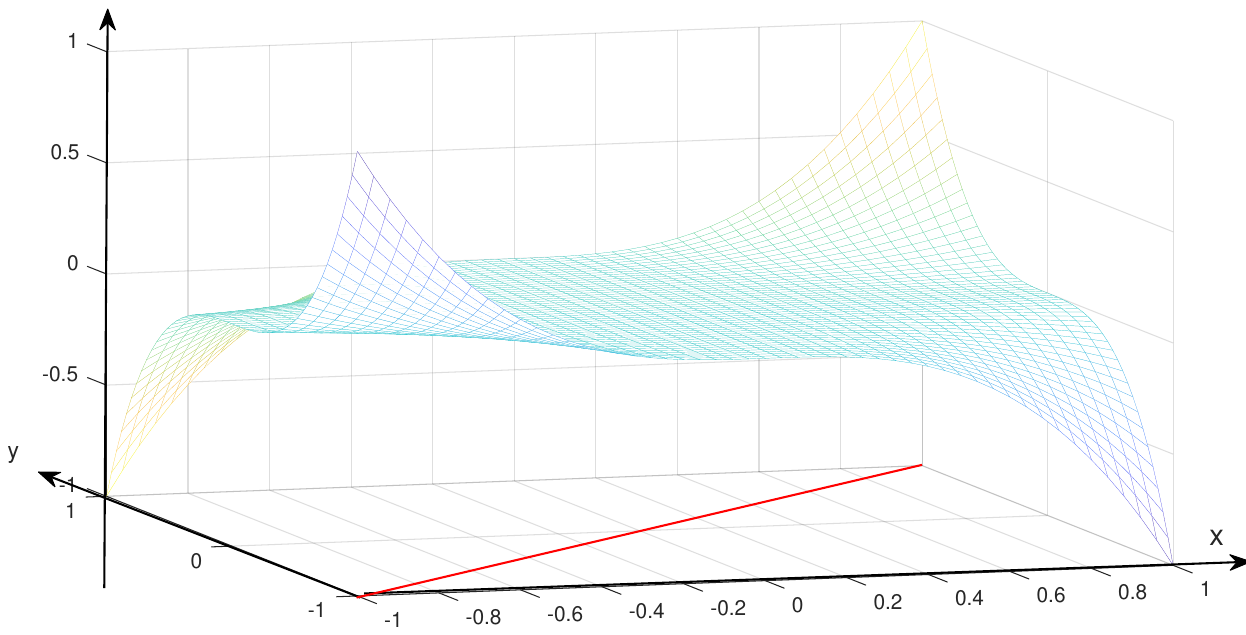}
\caption{Weyl's Kernel mesh for $d=3$ and $n=1$.}
\end{figure}

\begin{proposition} Let $d\in\NN_0$ and $n\in\NN$.

\nr{1} Weyl's kernel $W^{(d)}$ is symmetric and positive semidefinite.

\nr{2} The reproducing kernel Hilbert space of Weyl's kernel is the vector space of all $d$-homogeneous 
polynomial functions on $\RR^n$, that is,
\begin{equation*}
\cH_{W^{(d)}}(\RR^n)=\{f\colon \RR^n\ra \RR\mid f(x)
=\sum_{\alpha\in\NN_0^n(d)}w_\alpha x^\alpha,\ w_\alpha\in\RR\},
\end{equation*}
 and its dimension is
\begin{equation*}
\dim_\RR(\cH_{W^{(d)}}(\RR^n))={n+d-1\choose n-1}=\frac{(n+d-1)!}{(n-1)!d!}=:N.
\end{equation*}

\nr{3} The inner product of Weyl's reproducing kernel Hilbert space is
\begin{equation*}
\langle f,g\rangle_{W^{(d)}(\RR^n)}=\sum_{\alpha\in\NN_0^n(d)} \frac{\alpha!}{d!}w_\alpha v_\alpha,
\end{equation*}
for any
\begin{equation*}
f(x)=\sum_{\alpha\in\NN_0^n(d)}w_\alpha x^\alpha,\quad g(x)=\sum_{\alpha\in\NN_0^n(d)}v_\alpha x^\alpha.
\end{equation*}

\nr{4} An orthonormal basis of $\cH_{W^{(d)}}(\RR^n)$ is 
\begin{equation*}
\biggl\{\sqrt{\frac{\alpha!}{d!}} x^\alpha\mid \alpha\in\NN_0^n(d)\biggr\}.
\end{equation*}
\end{proposition}

\begin{proof} The idea is that, first we get a minimal linearisation of Weyl's kernel and then we 
use the idea from the proof of Moore's Theorem.
We consider the vector space of all $d$-homogeneous polynomial functions on 
$\RR^n$, that is,
\begin{equation*}
\cH_d(\RR^n)=\{f\colon \RR^n\ra \RR\mid f(x)
=\sum_{\alpha\in\NN_0^n(d)}w_\alpha x^\alpha,\ w_\alpha\in\RR\},
\end{equation*}
of dimension
\begin{equation*}
\dim_\RR(\cH_d(\RR^n))={n+d-1\choose n-1}=\frac{(n+d-1)!}{(n-1)!d!}=:N.
\end{equation*}

We canonically identify $\RR^N$ with $\cH_d(\RR^n)$ and by this identification, we make $\cH_d(\RR^n)$ an
inner product space by letting
\begin{equation*}
\langle f,g\rangle_{\cH_d(\RR^n)}=\sum_{\alpha\in\NN_0^n(d)} 
w_\alpha v_\alpha,
\end{equation*}
for any
\begin{equation*}
f(x)=\sum_{\alpha\in\NN_0^n(d)}w_\alpha x^\alpha,\quad g(x)=\sum_{\alpha\in\NN_0^n(d)}v_\alpha x^\alpha.
\end{equation*}
In this way, $\RR^N$ and $\cH_d(\RR^n)$ are identified as Hilbert spaces as well.

Let $V\colon \RR^n\ra\cH_d(\RR^n)$ be defined, with the identification of $\RR^N$ and $\cH_d(\RR^n)$ as 
explained before, by
\begin{equation*}
V(x):=\bigl( (C^\alpha_d)^{1/2}x^\alpha\bigr)_{\alpha\in\NN_0^n(d)}=\bigl(\sqrt{\frac{d!}{\alpha!}}x^\alpha\bigr), 
\quad x\in\RR^n,
\end{equation*}
and observe that
\begin{align*}\langle V(x),V(y)\rangle_{\cH_d(\RR^n)} & = \sum_{\alpha\in\NN_0^n(d)} (C^\alpha_d)^{1/2} x^\alpha 
(C^\alpha_d)^{1/2}y^\alpha \\
& = \sum_{\alpha\in\NN_0^n(d)} C^\alpha_d x^\alpha y^\alpha=W^{(d)}(x,y),\quad x,y\in\RR^n.
\end{align*}
Then, for any $m\in\NN$ and any $x^{(1)},\ldots, x^{(m)}\in\RR^n$ we have
\begin{equation*}
\biggl[ W^{(d)}(x^{(i)},x^{(j)})\biggr]_{i,j=1}^m =\biggl[ \langle V(x^{(i)}),V(x^{(j)})\rangle_{\cH_d(\RR^n)}\biggr]_{i,j=1}^m \geq 0
\end{equation*}
since the latter matrix is a Gram matrix.
The linearisation $(\cH_d(\RR^n);V)$ is actually minimal but we can see this in a more direct way.

Recall that, for each $y\in\RR^n$ we can consider the functions $W^{(d)}_y\colon \RR^n\ra\RR$ 
defined by
\begin{equation}\label{e:wede}
W^{(d)}_y(x):=W^{(d)}(x,y)=\sum_{\alpha\in\NN_0^n(d)} \frac{d!}{\alpha!} y^\alpha x^\alpha,\quad x\in\RR^n,
\end{equation}
and hence
\begin{equation*}
W_y^{(d)}\in\Span\{x^\alpha\mid \alpha\in \NN_0^n(d)\}=\cH_d(\RR^n).
\end{equation*}
But, $\Span\{W_y^{(d)}\mid y\in\RR^n\}$ is finite dimensional and dense in the reproducing kernel Hilbert space $\cH_{W^{(d)}}(\RR^n)$,
hence
\begin{equation*}
\cH_{W^{(d)}}(\RR^n)=\cH_d(\RR^n)=\Span\{x^\alpha\mid \alpha\in\NN_0^n\}.
\end{equation*}

However, this is not an identification of Hilbert spaces since the 
inner product on the reproducing kernel Hilbert space $\cH_{W^{(d)}}(\RR^n)$ is given by 
\begin{equation}\label{e:lafegew}
\langle f,g\rangle_{W^{(d)}(\RR^n)}=\sum_{\alpha\in\NN_0^n(d)} \frac{\alpha!}{d!}w_\alpha v_\alpha,
\end{equation}
for any
\begin{equation*}
f(x)=\sum_{\alpha\in\NN_0^n(d)}w_\alpha x^\alpha,\quad g(x)=\sum_{\alpha\in\NN_0^n(d)}v_\alpha x^\alpha.
\end{equation*}
Indeed, in order to prove that \eqref{e:lafegew} is correct it is sufficient to check it for $f=W^{(d)}_x$ and $g=W^{(d)}_y$ for 
arbitrary $x,y\in \RR^n$, which holds since, in view of \eqref{e:wede}, we have
\begin{equation*}
\langle W^{(d)}_x,W^{(d)}_y\rangle_{W^{(d)}(\RR^n)}=W^{(d)}(y,x)= \sum_{\alpha\in\NN_0^n(d)} \frac{\alpha!}{d!} \bigl(\frac{d!}{\alpha!}
x^\alpha\bigr) \bigl(\frac{d!}{\alpha!}y^\alpha\bigr)=  \sum_{\alpha\in\NN_0^n(d)} \frac{d!}{\alpha!}x^\alpha y^\alpha.
\end{equation*}

Finally, we can verify that, for any $\alpha,\beta\in\NN_0^n(d)$ we have
\begin{equation*}
\langle \sqrt{\frac{\alpha!}{d!}} x^\alpha,\sqrt{\frac{\beta!}{d!}} x^\beta\rangle_{W^{(d)}(\RR^n)} 
=\begin{cases} 1, & \alpha=\beta, \\
0, & \alpha\neq \beta,\end{cases}
\end{equation*}
hence 
\begin{equation*}
\biggl\{\sqrt{\frac{\alpha!}{d!}} x^\alpha\mid \alpha\in\NN_0^n(d)\biggr\},
\end{equation*}
is orthonormal and, from the comparison of dimensions, it follows that actually this is an orthonormal basis.\end{proof}

\subsection{Gaussian Kernels}
For $\sigma>0$ and $n\in\NN$ the \emph{Gaussian Kernel} $G^\sigma\colon X\times X\ra\RR$, where 
$X\subseteq \RR^n$ is some nonempty set, is
\begin{equation*}
G^\sigma(x,y):=\exp\bigl(-\frac{\|x-y\|^2}{\sigma^2}\bigr),\quad x,y\in X.
\end{equation*}

The Gaussian kernel is related to the \emph{Normal Distribution}, or the \emph{Gauss Distribution},
\begin{equation*}
\frac{1}{\sigma^n (2\pi)^{n/2}} \exp\bigl(\frac{-\|x-\mu\|^2}{2\sigma^2}\bigr),\quad x\in\RR^n,
\end{equation*}
where $\mu\in\RR^n$ is the \emph{Expected Value} and $\sigma$ is the \emph{Standard Deviation}, that
shows up naturally in the Central Limit Theorem and many (most of the) natural processes.

\begin{figure}[h]
\includegraphics[width=14cm, trim = 0cm 8cm 0cm 7cm]{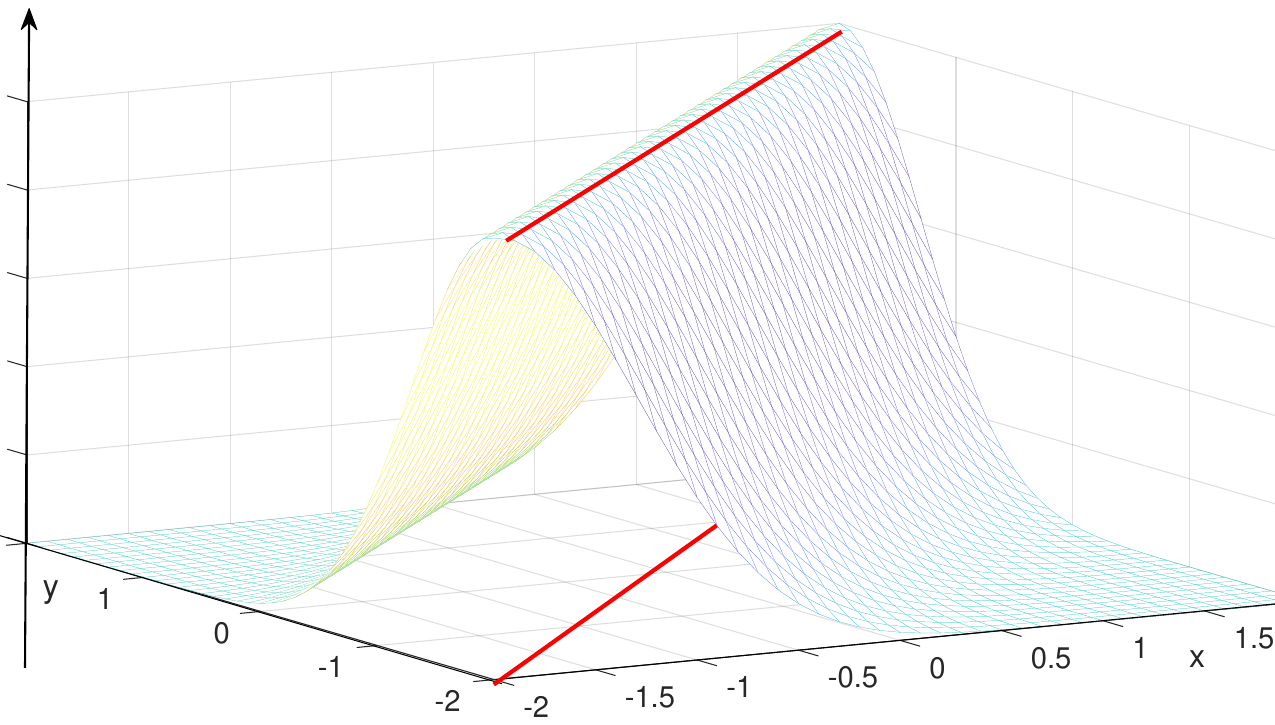}
\caption{The Gauss Kernel for $\sigma=1$ and $n=1$.}
\end{figure}

\begin{remark} Let $K\colon X\times X\ra\RR$ be a positive semidefinite kernel on $X\neq\emptyset$ and $f\colon X\ra\RR$ 
a function. Then, the kernel
\begin{equation*}
X\times X\ni (x,y)\mapsto f(x)K(x,y)f(y)\in\RR
\end{equation*}
is positive semidefinite. 

Indeed, for any $x_1,\ldots,x_m\in X$ we have
\begin{align*}
\bigl[f(x_i)K(x_i,x_j)f(x_j)\bigr]_{i,j=1}^m & = \Diag(f(x_i))_{i=1}^m \bigl[K(x_i,x_j)\bigr]_{i,j=1}^m
\Diag(f(x_i))_{i=1}^m \\
&=B^*AB\geq 0.
\end{align*}
\end{remark}

\begin{theorem} Let $X\subseteq\RR^n$ be a set with nonempty interior, $\sigma>0$ and $n\in\NN$, and consider the 
Gaussian kernel $G^\sigma\colon X\times X\ra\RR$, $G^\sigma(x,y)=\exp(-\|x-y\|^2/\sigma^2)$, $x,y\in X$.

\emph{(1)} The Gaussian kernel $G^\sigma$ is symmetric and positive semidefinite.

\nr{2} The reproducing kernel Hilbert space $\cH_{G^\sigma}(\RR^n)$ consists of all functions $f\colon X\ra\RR$ having representation
\begin{equation}\label{e:fexeb}
f(x)=\exp\bigl(-\frac{\|x\|^2}{\sigma^2}\bigr) \sum_{k=0}^\infty 
 \sum_{\alpha\in\NN_0^n(k)} w_\alpha x^\alpha,\quad 
x\in X,
\end{equation}
such that
\begin{equation}\label{e:sukoi}
\sum_{k=0}^\infty \frac{\sigma^{2k}}{2^k } \sum_{\alpha\in \NN_0^n(k)} \alpha! w_\alpha^2<\infty.
\end{equation}
The convergence in the series representation of $f$ is at least pointwise.

 \nr{3} The inner product of the reproducing kernel Hilbert space $\cH_{G^\sigma}(\RR^n)$ is
\begin{equation}\label{e:lafeger}
\langle f,g\rangle_{G^\sigma}=\sum_{k=0}^\infty \frac{\sigma^{2k}}{2^k}\sum_{\alpha\in\NN_0^n(k)} 
\alpha! w_\alpha v_\alpha,
\end{equation}
where $f,g\in\cH_{G^\sigma}(\RR^n)$ have representations as in \eqref{e:fexeb} and
\begin{equation}\label{e:gexe}
g(x)=\exp\bigl(-\frac{\|x\|^2}{\sigma^2}\bigr) \sum_{k=0}^\infty \sum_{\alpha\in\NN_0^n(k)} v_\alpha x^\alpha,\quad x\in X.
\end{equation}

\nr{4} For any $\alpha\in \NN_0^n$ let $\phi_\alpha\colon X\ra\RR$ be defined by
\begin{equation}\label{e:phialex}
\phi_\alpha(x):=\frac{2^{\frac{|\alpha|}{2}}}{\sigma^{|\alpha|}\sqrt{\alpha!}} \exp\bigl(-\frac{\|x\|^2}{\sigma^2}\bigr) x^\alpha,\quad x\in X.
\end{equation}
Then $\{\phi_\alpha\mid \alpha\in \NN_0^n\}$
is an orthonormal basis of $\cH_{G^\sigma}(\RR^n)$.
\end{theorem}

\begin{proof} For any $x,y\in X$ we have
\begin{align}\label{e:gesimixy}
G^\sigma(x,y) & = \exp\bigl(-\frac{\|x-y\|^2}{\sigma^2}\bigr) = 
\exp\bigl(-\frac{\langle x-y,x-y \rangle}{\sigma^2}\bigr) \\
& = \exp\bigl(-\frac{\|x\|^2-2\langle x,y\rangle +\|y\|^2}{\sigma^2}\bigr)\nonumber \\
& = \exp\bigl(-\frac{\|x\|^2}{\sigma^2}\bigr) \exp\bigl(2\frac{\langle x,y\rangle}{\sigma^2}\bigr) 
\exp\bigl(-\frac{\|y\|^2}{\sigma^2}\bigr) \nonumber \\
& = \sum_{k=0}^\infty \frac{2^k}{k!\sigma^{2k}} \exp\bigl(-\frac{\|x\|^2}{\sigma^2}\bigr)  \langle x,y\rangle^k
\exp\bigl(-\frac{\|y\|^2}{\sigma^2}\bigr).\nonumber
\end{align}
Then use the fact that Weyl's kernels are all positive semidefinite, the previous remark, and the fact that the 
convex cone of positive semidefinite kernels is closed under pointwise convergence, see 
Remark~\ref{r:schwartz}, and conclude that $G^\sigma$ is a positive semidefinite kernel. In addition, from \eqref{e:gesimixy} and
 \eqref{e:wedexy}, for any, $x,y\in \RR^n$ we have
\begin{align}\label{e:gesixy}
G^\sigma(x,y) & =  \exp\bigl(-\frac{\|x\|^2}{2}\bigr) \sum_{k=0}^\infty \frac{2^k}{\sigma^{2k}}   \sum_{\alpha\in\NN_0^n(k)} \frac{1}
{\alpha!} x^\alpha y^\alpha\exp\bigl(-\frac{\|y\|^2}{2}\bigr) \\
& = \sum_{k=0}^\infty\sum_{\alpha\in\NN_0^n(k)} \biggl(\frac{2^{k/2}}{\sigma^k \sqrt{\alpha!}}  
\exp\bigl(-\frac{\|x\|^2}{2}\bigr) x^\alpha \biggr) 
\biggl( \frac{2^{k/2}}{\sigma^k \sqrt{\alpha!}}  \exp\bigl(-\frac{\|y\|^2}{2}\bigr) y^\alpha\biggr).\nonumber
\end{align}

Let $\cH_0$ denote the space of all functions $f\colon X\ra \RR$ with representation as in \eqref{e:fexeb} and under the 
condition \eqref{e:sukoi}. It is easy to see that the series in \eqref{e:fexeb} is pointwise convergent under the condition 
\eqref{e:sukoi}. Also, if $f$ and $g$ are two functions with representations as in \eqref{e:fexeb} and \eqref{e:gexe}, then by the 
Cauchy Inequality it follows that the series in the right side of \eqref{e:lafeger} is pointwise convergent and it is a simple exercise 
to see that it defines an inner product with respect to which $\cH_0$ is a Hilbert space.

Since $X$ is supposed to have nonempty interior, it follows that the homogeneous polynomials $x^\alpha$, for 
$\alpha\in \NN_0^n$, are distinct. In order to see that the system of functions $\{\phi_\alpha\mid \alpha\in \NN_0^n\}$
is an orthonormal basis of $\cH_0$, we firstly observe that if $\alpha,\beta\in\NN^n$ are distinct then from \eqref{e:lafeger} it 
follows that $\langle \phi_\alpha,\phi_\beta\rangle{G^\sigma}=0$ while, for $\alpha=\beta$, letting $k=|\alpha|$, we have
\begin{equation*}
\langle \phi_\alpha,\phi_\alpha\rangle_{G^\sigma} = \frac{\sigma^{2k}\alpha!}{2^k} \frac{2^{\frac{|\alpha|}{2}}}{\sigma^{|\alpha|}
\sqrt{\alpha!}} \frac{2^{\frac{|\alpha|}{2}}}{\sigma^{|\alpha|}\sqrt{\alpha!}}=1.
\end{equation*} 
This proves that the system of functions $\{\phi_\alpha\mid \alpha\in \NN_0^n\}$ is orthonormal. From the definition of $\cH_0$, it 
follows that it is also total in $\cH_0$.

Finally, from \eqref{e:phialex} and \eqref{e:gesixy} it follows that
\begin{equation*}
G^\sigma(x,y)=\exp(-\|x-y\|/\sigma^2)= \sum_{k=0}^\infty \sum_{|\alpha|=k}\phi_\alpha(x)\phi_\alpha(y),\quad x,y\in X,
\end{equation*}
and
\begin{equation*}
1=G^\sigma(x,x)=  \sum_{k=0}^\infty \sum_{|\alpha|=k}|\phi_\alpha(x)|^2<\infty,\quad x\in X,
\end{equation*}
hence, by Theorem~\ref{t:on} and the uniqueness of the reproducing kernel Hilbert space associated to a positive semidefinite 
kernel, see Proposition~\ref{p:unic}, it follows that $\cH_{G^\sigma(\RR^n)}=\cH_0$.
\end{proof}

\section{Mercer's Theorem}\label{s:mt}

Throughout this section $X$ will be a compact metric space. We denote
\begin{equation}\label{e:cefax}
\cC_\FF(X):=\{f\colon X\ra \FF\mid f\mbox{ continuous on }X\},
\end{equation}
which, when endowed with the sup-norm $\|\cdot\|_\infty$, becomes a Banach space.

Let $\nu$ be a finite Borel measure on $X$. Such a measure is always both inner and outer regular, see 
Theorem 5.2.1 in H.~Geiss and S.~Geiss \cite{Geiss2025}, but we will not need this here.
As usual, we denote
\begin{equation}\label{e:lefax}
\cL_\FF^2(X,\nu):=\{f\colon X\ra \FF\mid f\mbox{ is measurable and }\int_X|f(x)|^2\de \nu(x)<\infty\},
\end{equation}
with pre-inner product
\begin{equation}\label{e:lafeg}
\langle f,g\rangle_{L^2}:=\int_X f(x)\ol{g(x)}\de \nu(x),
\end{equation}
and note that
\begin{align*}
\cN_\FF(X,\nu) & :=\{f\colon X\ra \FF\mid f\mbox{ is null }\nu\mbox{-a.e.}\} \\
& =\{f\in \cL^2_\FF(X,\nu)\mid \int_X |f(x)|^2\de\nu(x)=0\}\\
& \  = \{f\in\cL_\FF^2(X,\nu)\mid \int_X f(x)
\ol{g(x)}\de\nu(x)=0\mbox{ for all }g\in \cL^2_\FF(X,\nu)\},
\end{align*}
is a vector subspace of $\cL^2_\FF(X,\nu)$. By factorisation we obtain
\begin{equation}\label{e:lefex} 
L^2_\FF(X,\nu):=\cL^2_\FF(X,\nu)/\cN_\FF(X,\nu),
\end{equation}
which becomes a Hilbert space with the inner product induced by $\langle\cdot,\cdot\rangle_{L^2}$.

Clearly, $\cC_\FF(X)$ is naturally included into the vector space $\cL^2_\FF(X,\nu)$. 
Let $E_X\colon \cC_\FF(X)\ra L^2(X,\nu)$ be the composition of the inclusion 
$\cC_\FF(X)\hookrightarrow \cL^2_\FF(X,\nu)$ with the canonical projection 
$\pi \colon \cL^2_\FF(X,\nu)\ra L^2(X,\nu)$. This is a bounded linear operator and
\begin{equation*}
\|E_X\|\leq \nu(X)^{1/2}.
\end{equation*}
The inclusion of $\cC_\FF(X)$ in $\cL^2(X,\nu)$ is dense, although $\cL^2(X,\nu)$ is only a seminormed 
space and not Hausdorff, in general. However, since the canonical projection $\pi$ is surjective and continuous,
this implies that $\pi(\cC_\FF(X))$ is dense in $L^2(X,\nu)$.

The support of the Borel measure $\nu$ is defined as
\begin{equation}\label{e:supen}
\supp(\nu):=\{x\in X\mid \nu(V)\neq 0\mbox{ for all open neighbourhoods }V\mbox{ of }x\}.
\end{equation}
In general,
if $\cC_\FF(X)\cap \cN_\FF(X,\nu)$ is not trivial, $E_X$ does not induce an embedding of $\cC_\FF(X)$ into 
$L^2(X,\nu)$. As can be easily seen, if $\supp(\nu)=X$ then $\cC_\FF(X)\cap \cN_\FF(X,\nu)=0$ and, 
in this case, we have an embedding $E_X:\cC_\FF(X)\hookrightarrow L^2(X,\nu)$.

\subsection{Integral Operators with Continuous Symbols}

\begin{proposition}\label{p:leka}
Assume that $K\colon X\times X\ra\FF$ is a continuous kernel and $\nu$ is a finite Borel measure on $X$.

\nr{i} The formula
\begin{equation}\label{e:leka}
(L_Kf)(x):=\int_X K(x,t)f(t)\de\nu(t)
\end{equation}
correctly defines a bounded linear operator $L_K\colon L^2_\FF(X,\nu)\ra L^2_\FF(X,\nu)$ and
\begin{equation}\label{e:lekal}
\|L_K\|\leq \biggl(\int_X\int_X|K(x,t)|^2\de\nu(x)\de\nu(t)\biggr)^{1/2}.
\end{equation}

\nr{ii} $\ran(L_K)\subseteq \cC_\FF(X)$ and, when considering the integral operator defined by \eqref{e:leka}
with codomain 
$\cC_\FF(X)$, denoted by $L_{K,\mathrm{c}}\colon L^2_\FF(X,\nu)\ra 
\cC_\FF(X)$, we have
\begin{equation}\label{e:lekac}
\|L_{K,\mathrm{c}}\|\leq \nu(X)^{1/2} \sup_{x,t\in X}|K(x,t)|,
\end{equation}
and, consequently,
\begin{equation}\label{e:lekale}
\|L_K\|\leq \nu(X)\sup_{x,t\in X}|K(x,t)|.
\end{equation}

\nr{iii} $L_{K,\mathrm{c}}$ and, consequently $L_K$, are compact operators.

\nr{iv} $L_K$ is a Hilbert-Shmidt operator and
\begin{equation}\label{e:hs}
\|L_K\|_{\mathrm{HS}}^2=\int_X\int_X|K(x,t)|^2\de\nu(x)\de\nu(t).
\end{equation}
\end{proposition}

\begin{proof} (i)
Let $f\in \cL^2_\FF(X,\nu)$ be arbitrary and note that $X\ni t\mapsto K(x,t)\in\FF$ is continuous, hence
in $\cL_\FF^2(X,\nu)$, which implies that $K(x,\cdot)f(\cdot)\in \cL^1_\FF(X,\nu)$, 
the integral in \eqref{e:leka} 
converges and $L_Kf$ is a measurable function. Also,
\begin{align*}
\|L_Kf\|_2^2 & = \int_X \bigl| \int_X K(x,t)f(t)\de\nu(t)\bigr|^2\de\nu(x)\\
\intertext{which, by the CBS Inequality is dominated by}
& \leq \int_X\bigl( \int_X |K(x,t)|^2\de\nu(t) \int_X|f(t)|^2\de\nu(t)\bigr)\de\nu(x)=\|f\|^2_2 \int_X \int_X |K(x,t)|^2\de\nu(t) \de\nu(x),
\end{align*}
which proves that $L_Kf\in\cL^2_\FF(X,\nu)$. It is a standard argument to see that $L_Kf$, when considered as 
an element in $L^2_\FF(X,\nu)$, does not depend on the representative of $f\in L^2_\FF(X,\nu)$. Hence, 
$L_K\colon L^2_\FF(X,\nu)\ra L^2_\FF(X,\nu)$ is a linear bounded operator and its operator norm is dominated as in
\eqref{e:lekal}.

(ii) For each $x_1,x_2\in X$ we have
\begin{align}
|(L_Kf)(x_1)-(L_Kf)(x_2)| & = \bigl| \int_X \bigl( K(x_1,t)-K(x_2,t)\bigr)f(t)\de \nu(t)\bigr|\nonumber \\
\intertext{which, by the CBS Inequality, is dominated by}
& \leq \bigl( \int_X |K(x_1,t)-K(x_2,t)|^2\de\nu(t)\bigr)^{1/2} \bigl( \int_X |f(t)|^2\de\nu(t)\bigr)^{1/2} \nonumber\\
& \leq \nu(X)^{1/2} \sup_{t\in X}|K(x_1,t)-K(x_2,t)| \|f\|_2.\label{e:lekax}
\end{align}
Since $K$ is continuous on the compact space $X\times X$ it follows that $K$ is uniformly continuous hence
$L_Kf\in\cC_\FF(X)$. It is a standard argument to see that $L_Kf$ as a function in $\cC_\FF(X)$ does not depend 
on the representative of $f\in L^2_\FF(X)$, hence we have a linear operator 
$L_{K,\mathrm{c}}\colon L^2_\FF(X)\ra \cC_\FF(X)$.

Similarly as before, for arbitrary $f\in \cL^2_\FF(X,\nu)$ we have
\begin{equation}\label{e:lekam}
|(L_{K,\mathrm{c}}f)(x)|\leq \nu(X)^{1/2}\sup_{t\in X}|K(x,t)| \|f\|_2,\quad x\in X,
\end{equation}
hence the inequalities \eqref{e:lekac} and \eqref{e:lekale} hold.

(iii) In order to prove that the operator $L_{K,\mathrm{c}}$ is compact we have to prove that for any bounded 
sequence $(f_n)_n$ in $L^2_\FF(X,\nu)$ there exists a subsequence $(f_{k_n})_n$ such that the sequence 
$(L_{K,\mathrm{c}}f_{k_n})_n$ converges uniformly in $\cC_\FF(X)$. This is a bit more general than what we recall in 
Subsection~\ref{ss:co}.

Indeed, assuming that
\begin{equation*}
\|f_n\|_2\leq C,\quad n\in\NN,
\end{equation*}
on the one hand, from \eqref{e:lekam} we get
\begin{equation*}
\|L_{K,\mathrm{c}}f_n\|_\infty \leq \nu(X)^{1/2} \|K\|_\infty \|f_n\|_2\leq \nu(X)^{1/2} \|K\|_\infty C,\quad n\in\NN,
\end{equation*}
hence the sequence $(L_{K,\mathrm{c}}f_{k_n})_n$ is uniformly bounded.

On the other hand, for each $x_1,x_2\in X$, from \eqref{e:lekax} we have
\begin{align*}
|(L_{K,\mathrm{c}}f_n)(x_1)-(L_{K,\mathrm{c}}f_n)(x_2)| & \leq \nu(X)^{1/2} \sup_{t\in X}|K(x_1,t)-K(x_2,t)| \|f_n\|_2 \\ & \leq \nu(X)^{1/2} \sup_{t\in X}|K(x_1,t)-K(x_2,t)| C,
\end{align*}
hence the sequence $(L_{K,\mathrm{c}}f_{k_n})_n$ is equicontinuous. By the 
Arzel\`a-Ascoli Theorem, the set
$\{L_{K,\mathrm{c}}f_{k_n}\mid n\in\NN\}$ is precompact in $\cC_\FF(X)$, hence there exists a subsequence $(f_{k_n})_n$ such 
that the sequence 
$(L_{K,\mathrm{c}}f_{k_n})_n$ converges uniformly in $\cC_\FF(X)$.

(iv) Since $X$ is a compact metric space and $\nu$ is a finite Borel measure it follows that the Hilbert space
$L^2_\FF(X,\nu)$ is separable. Let $(e_n)_n$ be an arbitrary orthonormal basis of $L^2_\FF(X,\nu)$. Then, for each 
$x\in X$, considering the function $K(x,\cdot)\in L^2_\FF(X,\nu)$, by Parseval's Identity we have
\begin{align}
\int_X |K(x,t)|^2\de\nu(t) & = \sum_{n=1}^\infty |\langle K(x,\cdot),e_n\rangle_{L^2_\FF(X,\nu)}|^2 \nonumber \\
& = \sum_{n=1}^\infty \bigl| \int_X K(x,t)\ol{e_n(t)}\de\nu(t)\bigr|^2 \nonumber\\
& = \sum_{n=1}^\infty |(L_K\ol{e_n})(x)|^2.\label{e:intexka}
\end{align}
Note that both sides of \eqref{e:intexka} are made of continuous functions with respect to the variable $x$ and hence we can integrate over $X$ with respect to the finite Borel measure $\nu$ and get
\begin{align}
\int_X\int_X |K(x,y)|^2\de\nu(y)\de\nu(x)  & = \int_X \sum_{n=1}^\infty |(L_K\ol{e_n})(x)|^2\de\nu(x)\nonumber \\
\intertext{which, by applying the Bounded Convergence Theorem, equals}
& = \sum_{n=1}^\infty \int_X |(L_K\ol{e_n})(x)|^2\de\nu(x) \nonumber \\
& = \sum_{n=1}^\infty \|L_K\ol{e_n}\|_{L^2_\FF(X,\nu)}^2<\infty.\label{e:intexx}
\end{align}
But the sequence $(\ol{e_n})_n$ is an orthonormal basis of $L^2_\FF(X,\nu)$ as well and hence, this means that the 
operator $L_K$ is a Hilbert-Schimdt operator and its Hilbert-Schmidt norm is given by the formula \eqref{e:hs}, see 
Subsection~\ref{ss:hso}.
\end{proof}

The operator $L_K$ defined as in \eqref{e:leka} is called the \emph{integral operator} 
with \emph{symbol} $K$.

\begin{remark}\label{r:hs} Proposition~\ref{p:leka} is a special case of a more general theory, that of D.~Hilbert 
and E.~Schmidt on square integrable kernels. With a similar proof to that of the item (iv), one can prove that if 
$K\in L^2_\FF(X\times X,\nu\times \nu)$ then the formula \eqref{e:leka} provides a bounded linear operator 
$L_K\colon L^2_\FF(X,\nu)\ra L^2_\FF(X,\nu)$ which is a Hilbert-Schmidt operator, in particular it is compact, and such 
that \eqref{e:hs} holds as well hence we have a linear transformation
\begin{equation}\label{e:hsmap}
L^2_\FF(X\times X,\nu\times \nu)\ni K\mapsto L_K\in\cB_2(L^2_\FF(X,\nu)),
\end{equation}
where $\cB_2(L^2_\FF(X,\nu))$ denotes the Hilbert-Schmidt ideal of operators on the Hilbert space $L^2_\FF(X,\nu)$,
and this linear transformation is isometric. In particular, 
the correspondence between $K$ and $L_K$ is one-to-one, in the sense that 
the kernels are unique $\nu$-a.e. According to a celebrated theorem of D.~Hilbert and E.~Schmidt, this map is 
surjective, that is, to each Hilbert-Schmidt operator $T\in \cB_2(L^2_\FF(X,\nu))$ there corresponds a kernel
$K\in L^2_\FF(X\times X,\nu\times \nu)$ such that $T=L_K$, that is, $T$ is the integral operator with symbol $K$, 
hence we have \eqref{e:hsmap} is a unitary operator. 
Recall that 
$\cB_2(L^2_\FF(X,\nu))$ is a Hilbert space under the inner product
\begin{equation*}
\langle T,S\rangle_{\mathrm{HS}}=\tr(T^*S),\quad T,S\in \cB_2(L^2_\FF(X,\nu)).
\end{equation*}
To put this in a different perspective, 
recall that, for our setting, of a compact metric space $X$ and a finite Borel measure 
$\nu$ on $X$, the Hilbert space $L^2_\FF(X,\nu)$ is separable and so is $L^2_\FF(X\times X,\nu\times \nu)$. Actually, 
there is a canonical identification of $L^2_\FF(X\times X,\nu\times \nu)$ with the tensor product 
$L^2_\FF(X,\nu)\otimes L^2_\FF(X,\nu)$, given by $(\phi\otimes \psi)(x,y)=\phi(x)\psi(y)$. 

Now, to tackle the surjectivity of the map \eqref{e:hsmap},
if $T\in\cB_2(L^2_\FF(X,\nu))$, then $T=A+\iac B$ with $A=A^*\in\cB_2(L^2_\FF(X,\nu))$ and 
$B=B^*\in\cB_2(L^2_\FF(X,\nu))$, hence it is sufficient to find the symbol of $T$ under the extra assumption that 
$T=T^*$. Then, by the spectral theory of compact operators, see Subsection~\ref{ss:co}, there exists an orthonormal basis 
$\{\phi_n\mid n\in\NN\}$ of $L^2_\FF(X,\nu)$ made up by eigenfunctions of $T$ and such that 
$T\phi_n=\lambda_n \phi_n$, where the eigenvalues $\lambda_n$ are counted according to their multiplicities. 
Then, one takes the kernel $K$ defined by
\begin{equation}\label{e:kaxy}
K(x,y):=\sum_{n=1}^\infty \lambda_n \phi_n(x)\ol{\phi_n(y)},\quad x,y\in X,
\end{equation}
and by a straightforward calculation 
it can be proven that $T=L_K$. The convergence of the series is guaranteed in 
$L^2_\FF(X\times X,\nu\times\nu)$. 
\end{remark}

\begin{remark}\label{r:riemann} If $\nu$ is a finite Borel measure on the compact 
metric space $X$ and $\phi\in L^1_\FF(X,\nu)$ 
is bounded, then $\int_X \phi(x)\de\nu(x)$ can be approximated by 
Riemann sums $R(\phi;\cP;\bx)$ associated to tagged measurable partitions $(\cP;\bx)$, 
that is, $\cP=\{A_i\}_{i=1}^n$ is a 
measurable partition of $X$ and $\bx:=\{x_1,\ldots,x_n\}$, with $x_i\in A_i$, for all $i=1,\ldots,n$, defined by
\begin{equation*}
R(\phi;\cP;\bx):=\sum_{i=1}^n f(x_i)\nu(A_i).
\end{equation*}
This is because, without loss of generality, we can assume that $\phi$ is real and then 
$\int_X \phi(x)\de\nu(x)$ can be approximated by Darboux sums $D_\mathrm{min}(\phi;\cP)$ 
and $D_\mathrm{max}(\phi;\cP)$, where 
\begin{equation*}
D_\mathrm{min}(\phi;\cP):=\sum_{i=1}^n m_i \nu(A_i),\quad D_\mathrm{max}(\phi;\cP):=\sum_{i=1}^n M_i \nu(A_i),
\end{equation*} where $\cP=\{A_i\}_{i=1}^n$ is a measurable partition of $X$, 
$m_i=\inf_{x\in A_i} \phi(x)$ and $M_i:=\sup_{x\in A_i} \phi(x)$, for all $i=1,\ldots,n$, e.g.\ see Chapter 12 in 
\cite{Baer}, and then, for any 
$\bx:=\{x_1,\ldots,x_n\}$ with $x_i\in A_i$, for all $i=1,\ldots,n$, we have
\begin{equation*}
D_\mathrm{min}(\phi;\cP) \leq R(\phi;\cP;\bx)\leq D_\mathrm{max}(\phi;\cP).
\end{equation*}
\end{remark}

\begin{proposition}\label{p:lekad} 
Let $\nu$ be a finite Borel measure on the compact metric space $X$ and $K\in\cK_\FF(X)$
a continuous kernel. Consider the bounded linear operator $L_K\colon L^2_\FF(X,\nu)\ra L^2_\FF(X,\nu)$ defined by
\eqref{e:leka}.

\nr{i} $L_K^*=L_{K^*}$, where $K^*(x,y)=\ol{K(y,x)}$ for all $x,y\in X$.

\nr{ii} If $K$ is Hermitian/symmetric then $L_K=L_K^*$.

\nr{iii} If $K$ is Hermitian/symmetric and positive semidefinite then $L_K\geq 0$, in the sense that
$\langle L_K,f,f\rangle_2\geq 0$ for all $f\in L^2_\FF(X,\nu)$.
\end{proposition}

\begin{proof} (i) Fubini's Theorem is applicable and, for any $f,g\in L^2_\FF(X,\nu)$ we have
\begin{align*}\langle L_Kf,g\rangle_2 & = \int_X\bigl(\int_X K(x,y)f(y)\de\nu(y)\bigr)\ol{g(x)}\de\nu(x)\\ 
& = \int_X \int_X K(x,y)f(y)\ol{g(x)}\de\nu(y)\de\nu(x) \\
& = \int_X\int_X f(y) \ol{\ol{K(x,y)}g(x)}\de\nu(y)\de\nu(x) \\
& = \int_X \int_X f(y) \ol{K^*(y,x)g(x)}\de\nu(x)\de\nu(y) = \langle f,L_{K^*}g\rangle_2.
\end{align*}

(ii) Consequence of (i).

(iii)  We have to show that, for any $f\in\cL^2_\FF(X,\nu)$ we have
\begin{equation}\label{e:lekaf}
\langle L_Kf,f\rangle_2=\int_X\int_X K(x,t)f(t)\ol{f(x)}\de\nu(t)\de\nu(x)\geq 0.
\end{equation}
Since $f\in \cL_\FF^2(X,\nu)$, it can be approximated with respect to $\|\cdot\|_{L^2}$ by bounded functions
in $\cL_\FF^2(X,\nu)$, hence, without loss of generality, it is sufficient to prove \eqref{e:lekaf} with the
extra assumption that $f$ is bounded. In this case, we use Remark~\ref{r:riemann} and the observation that,
for any measurable partition $\{A_i\}_{i\in\NN}$ of $X$ and any choice of points $x_i\in A_i$, $i=1,\ldots,n$,
the corresponding Riemann sums 
\begin{equation*}
\sum_{i,j=1}^n K(x_i,x_j) f(x_j)\ol{f(x_i)} \nu(A_i)\nu(A_j)= \sum_{i,j=1} K(x_i,x_j)\alpha_j \ol{\alpha_i}\geq 0
\end{equation*}
 where
$\alpha_i=\nu(A_i)f(x_i)$, for all $i=1,\ldots,n$.
\end{proof}

Given $K\in\cK_\FF(X)$, the \emph{support} of $K$, denoted by $\supp(K)$, 
is the collection of all $x\in X$ subject 
to the condition that there exists $y\in X$ such that either $K(x,y)\neq 0$ or $K(y,x)\neq 0$. Also, by 
Lemma~\ref{l:schwarz}, if $K$ is Hermitian/symmetric and positive semidefinite then, for a fixed $x\in X$ we 
have $K(x,x)=0$ if and only if $K(x,y)=0$ and $K(y,x)=0$ for all $y\in X$, hence 
$\supp(K)$ consists in all those $x\in X$ such that $K(x,x)\neq 0$.

The next proposition provides a converse to the statement from
item (iii) in the previous proposition, that is, the 
positivity of the integral operator $L_K$ implies that the kernel $K$ is positive semidefinite, if the measure $\nu$ 
contains sufficient information about the kernel $K$, mathematically illustrated by the fact that the support of 
$\nu$ contains the support of $K$.

\begin{proposition}\label{p:conv} Given $K\in \cK_\FF^{\mathrm{H}}\cF(X)$, 
assume that the finite Borel measure $\nu$ on the compact metric 
space $X$ has the property that $\supp(\nu)\supseteq \supp(K)$. If $L_K\geq 0$ then $K$ is positive semidefinite.
\end{proposition}

\begin{proof} Let $x_1,\ldots,x_n\in X$ and $\alpha_1,\ldots,\alpha_n\in\FF$ be arbitrary. We want to prove that
\begin{equation}\label{e:sijen}
\sum_{i,j=1}^n K(x_i,x_j)\alpha_j\ol{\alpha_i}\geq 0.
\end{equation}
To this end, we first observe that, without loss of generality, we can assume that $x_i\in\supp(K)$ for all 
$i=1,\ldots, n$. Indeed, if $x_l\not\in\supp(K)$ then $K(x_l,x_j)=K(x_j,x_l)=0$ for all $j=1,\ldots,n$ and hence 
the double sum in \eqref{e:sijen} can exclude $i=l$ and $j=l$ without changing anything.

Also, for the moment we assume that the points $x_1,\ldots,x_n$ are distinct.
Let $\epsilon>0$ be arbitrary. Since, for each $i=1,\ldots,n$, in view 
of the continuity of $K$, there exists $U_i$ an open neighbourhood of $x_i$ such that, for all $i,j=1,\ldots,n$, we 
have
\begin{equation}\label{e:kaxyd}
|K(x,y)-K(x_i,x_j)|<\epsilon,\quad x\in U_i,\ y\in U_j.
\end{equation}
Because the points $x_1,\ldots,x_n$ are distinct, we can find the collection of open sets $\{U_i\mid i=1,\ldots,n\}$ 
mutually disjoint.
Since, for each $i=1,\ldots,n$ we have $x_i\in\supp(K)\subseteq\supp(\nu)$ it follows that $\nu(U_i)=r_i>0$ and 
then, from \eqref{e:kaxyd}, for all $i,j=1,\ldots,n$, we get
\begin{equation}\label{e:intui}
\int_{U_i}\int_{U_j} |K(x,y)-K(x_i,x_j)|\de\nu(x)\de\nu(y)<\epsilon r_i r_j.
\end{equation}
Let
\begin{equation*}
g:=\sum_{i=1}^n \frac{\alpha_i}{r_i}\chi_{U_i}\in L^2_\FF(X,\nu).
\end{equation*}
Then,
\begin{align*}
\bigl| \langle L_Kg,g\rangle_2 & -\sum_{i,j=1}^n \ol{\alpha_i}\alpha_jK(x_i,x_j)\bigr| \\
 & = 
\bigl| \int_X\int_X \sum_{i,j=1}^n \frac{\ol{\alpha_i}\alpha_j}{r_i r_j} K(x,t)\chi_{U_j}(t)\chi_{U_i}(x)\de\nu(t)\de\nu(x)
\!-\! \!\sum_{i,j=1}^n \ol{\alpha_i}\alpha_j K(x_i,x_j)\bigr| \\
\intertext{which, in view of the fact that the neighbuourhood $U_i$ are mutually disjoint, equals}
& = \bigl| \sum_{i,j=1}^n  \frac{\ol{\alpha_i}\alpha_j}{r_i r_j} \int_{U_i}\int_{U_j} \bigl(K(x,t)-K(x_i,x_j)\bigr)\de\nu(t)\de\nu(x)\bigr| \\
& \leq \sum_{i,j=1}^n  \frac{|\alpha_i\alpha_j|}{r_i r_j} \int_{U_i}\int_{U_j} \bigl|K(x,t)-K(x_i,x_j) \bigr|\de\nu(t)\de\nu(t) \\
\intertext{which, in view of \eqref{e:intui}, is dominated by}
& \leq \epsilon \sum_{i,j=1}^n  |\alpha_i\alpha_j|.
\end{align*}
Since $\epsilon>0$ is arbitrary and $L_K\geq 0$, the inequality \eqref{e:sijen} is proven.

Finally, we have to deal with the case when the points $x_1,\ldots,x_n$ may not be distinct. In this case, 
if the set $x_1,\ldots,x_n$ may not consist of distinct points, the matrix $[K(x_i,x_j)]_{i,j=1}^n$ will have some 
identical columns and the corresponding identical rows. But $[K(x_i,x_j)]_{i,j=1}^n\geq 0$ if and only if all square 
principal determinants are nonnegative and we observe that, actually, 
only those principal determinants that do not contain identical rows or/and identical columns matter, 
and these correspond exactly to the case when the underlying set of points $\{x_{k_1},\ldots,x_{k_m}\}$ 
are distinct.
\end{proof}

\subsection{Mercer's Kernels}

\begin{definition}\label{d:mercer} Given a compact metric space $X$, a kernel $K\colon X\times X\ra\FF$ 
is called a \emph{Mercer kernel} if it is continuous, Hermitian/symmetric, and positive semidefinite.
\end{definition}

\begin{proposition} Let $\nu$ be a finite Borel measure with $\supp(\nu)=X$. Then, the mapping 
$K\mapsto L_K$, when $K$ runs through the set of all Mercer kernels on $X$, is one-to-one.
\end{proposition}

\begin{proof} We use Remark~\ref{r:hs} and the fact that, since $\supp(\nu)=X$, $\cC_\FF(X\times X)$ is
embedded into $L^2_\FF(X\times X,\nu\times \nu)$.
\end{proof}

If $K$ is a Mercer kernel, since it is Hermitian/symmetric and positive semidefinite, its reproducing kernel 
Hilbert space $\cH_K$ exists and is unique, by Theorem~\ref{t:moore} and Proposition~\ref{p:unic}.

\begin{proposition}\label{p:mercec}
If $K$ is a Mercer kernel on the compact metric space $X$ and $\cH_K$ is its reproducing kernel Hilbert space, 
then $\cH_K\subseteq \cC_\FF(X)$ and, letting $I_K\colon \cH_K\hookrightarrow \cC_\FF(X)$ denote the 
inclusion operator, we have $\|I_K\|\leq \sup_{x\in X}K(x,x)^{1/2}$, hence the inclusion of $\cH_K$ into 
$\cC_\FF(X)$ is continuous, equivalently, convergence in $\cH_K$ implies uniform convergence.
\end{proposition}

\begin{proof} If $x_1,x_2\in X$ are arbitrary and $f\in \cH_K$, then
\begin{align*}
|f(x_1)-f(x_2)|^2 & =|\langle f,K_{x_1}-K_{x_2}\rangle_{\cH_K}|^2 \leq 
\|f\|_{\cH_K}^2 \|K_{x_1}-K_{x_2}\|_{\cH_K}^2 \\
& = \|f\|^2_{\cH_K} \langle K_{x_1}-K_{x_2},K_{x_1}-K_{x_2}\rangle_{\cH_K} \\
& = \|f\|^2_{\cH_K} \bigl(K(x_1,x_1)-2\Re K(x_1,x_2)+K(x_2,x_2)\bigr),
\end{align*}
hence, the continuity of the kernel $K$ implies that $f$ is a continuous function on $X$.

Similarly, for arbitrary $f\in \cH_K$ and $x\in X$ we have
\begin{equation*}
|f(x)|=|\langle f,K_x\rangle_{\cH_K}|\leq \|f\|_{\cH_K} \|K_x\|_{\cH_K}=\|f\|_{\cH_K} K(x,x)^{1/2},
\end{equation*}
hence
\begin{equation*}
\|I_Kf\|_\infty =\sup_{x\in X}|f(x)| \leq \|f\|_{\cH_K} \sup_{x\in X}K(x,x)^{1/2},
\end{equation*}
which proves that $I_K$ is bounded and that $\|I_K\|\leq \sup_{x\in X}K(x,x)^{1/2}$.
\end{proof}

\begin{proposition}\label{p:sepa} 
If $K$ is a Mercer kernel on a compact metric space $X$ then its reproducing kernel Hilbert 
space $\cH_K$ is separable.
\end{proposition}

\begin{proof} Since $X$ is a compact metric space, it is separable. More precisely, for each $n\in \NN$ 
consider the open cover $\{B_{1/n}(x)\mid x\in X\}$ of $X$ from which we extract a finite cover 
$\{B_{1/n}(x_{n,j})\mid j=1,\ldots,m_n\}$ and note that the set
\begin{equation*}
D:=\bigcup_{n\in\NN} \{x_{n,1},\ldots, x_{n,m_n}\}
\end{equation*}
is countable and dense in $X$.

Then, we can prove that for any $x\in X$, letting $(x_n)_n$ be a sequence from $D$ that converges to $x$, it follows 
that $K_{x_n}\ra K_x$ as $n\ra\infty$ with respect to the norm $\|\cdot\|_{\cH_K}$. Indeed, in view of the 
continuity of $K$, we have
\begin{align*}
\|K_x-K_{x_n}\|^2_{\cH_K} & = \langle K_x-K_{x_n},K_x-K_{x_n}\rangle_{\cH_K} \\
& = K(x,x)-K(x,x_n)-K(x_n,x)+K(x_n,x_n)\ra 0\mbox{ as }n\ra\infty,
\end{align*}
From here it follows that the countable set $\FF_\QQ\{K_{x}\mid x\in D\}$, where $\FF_\QQ=\QQ$ if $\FF=\RR$ 
and $\FF_\QQ=\QQ+\iac \QQ$ if $\FF=\CC$, is countable and dense in $\cH_K$.
\end{proof}

\subsection{Reproducing Kernel Hilbert Spaces Associated to Mercer's Kernels}

We are now heading to Mercer's Theorem and the first step in this enterprise is to find a more explicit 
representation of the reproducing kernel Hilbert space $\cH_K$ in terms of the integral operator $L_K$. First we 
need to see that $\cH_K$ is dense in the support space of $L_K$.

\begin{remark}\label{r:total}
Let $g\in L^2_\FF(X,\nu)$ and $x\in X$ be arbitrary. Then, since $K$ is Hermitian we have
$\ol{K(t,x)}=K(x,t)$ for all $t\in X$ and then
\begin{equation*}
\langle g,K_x\rangle_{L^2_\FF}=\int_X g(t)\ol{K_x(t)}\de\nu(t)= \int_X g(t)\ol{K(t,x)}\de\nu(t)= \int_X K(x,t)g(t)
\de\nu(t)=(L_Kg)(x).
\end{equation*}
From here it follows that, if $g\perp K_x$ for all $x\in X$ if and only if $g\in \ker(L_K)$. This shows that 
$\{K_x\mid x\in X\}$ is total in $L^2_\FF(X,\nu)\ominus\ker(L_K)$, in particular $\cH_K$ is dense in 
$L^2_\FF(X,\nu)\ominus\ker(L_K)$.
\end{remark}

Before dealing with $\cH_K$, we show that the range of $L_K$ is a reproducing kernel Hilbert space and we
calculate its reproducing kernel. This is a preliminary technical step.

\begin{proposition}\label{p:kadoi}
Let $K$ be a Mercer kernel on the compact metric space $X$, let $\nu$ be a finite Borel 
measure on $X$ and let $L_K$ be the integral operator with symbol $K$. 
Define the kernel $K^{(2)}\colon X\times X\ra\FF$ by
\begin{equation}\label{e:kadoi}
K^{(2)}(x,y):=(L_K K_y)(x)=\int_X K(x,t)K(t,y)\de\nu(t),\ x,y\in X.
\end{equation}

\nr{1} $K^{(2)}$ is a Mercer kernel on $X$.

\nr{2} $\ran(L_K)=\cH_{K^{(2)}}$ in such a way that, when viewing 
$L_K\colon L^2_\FF(X,\nu)\ominus\ker(L_K)\ra \cH_{K^{(2)}}$, it is a unitary operator.

\nr{3} $L_K^2=L_{K^{(2)}}$, that is, $L_K^2$ is an integral operator with symbol $K^{(2)}$.

\nr{4} If $\supp(\nu)\supseteq \supp(K)$ then 
$K^{(2)}\leq \|L_K\|\, K$ and $\ran(L_K)=\cH_{K^{(2)}}\subseteq \cH_K$.
\end{proposition}

\begin{proof} (1) As in Proposition~\ref{p:leka} it follows that $K^{(2)}$ is continuous and since 
$K$ is Hermitian/symmetric it
follows easily that $K^{(2)}$ is Hermitian/symmetric as well. We prove that $K^{(2)}$ is positive semidefinite 
as well. To see this, let $x_1,\ldots,x_n\in X$ and $\alpha_1,\ldots,\alpha_n\in\FF$ be arbitrary. Then
\begin{align*}
\sum_{i,j=1}^n \ol{\alpha_i}\alpha_j K^{(2)}(x_i,x_j) & = \sum_{i,j=1}^n \ol{\alpha_i}\alpha_j \int_X K(x_i,t)K(t,x_j)\de\nu(t) \\
& = \int_X \sum_{i,j=1}^n \ol{\alpha_i}\alpha_j \ol{K(t,x_i)} K(t,x_j)\de\nu(t) \\
& = \int_X | \sum_{j=1}^n \alpha_j K(x_j,t)|^2\de\nu(t)\geq 0.
\end{align*}
This concludes the proof that $K^{(2)}$ is a Mercer kernel on $X$ as well.

(2) Let $\cR:=\ran(L_K)$ and note that $L_K|_{L^2_\FF(X,\nu)\ominus\ker(L_K)}$ is one-to-one and has the 
same range with $L_K$. On $\cR$ we define the inner product
\begin{equation*}
\langle L_Kf,L_Kg\rangle_\cR:= \langle f,g\rangle_{L^2_\FF(X,\nu)},\quad f,g\in L^2_\FF(X,\nu)\ominus \ker(L_K),
\end{equation*}
that is, we transport the Hilbert space structure from $L^2_\FF(X,\nu)\ominus \ker(L_K)$ to $\cR$, hence 
$(\cR,\langle\cdot,\cdot\rangle_\cR)$ is a Hilbert space and $L_K\colon L^2_\FF(X,\nu)\ominus\ker(L_K)\ra \cR$ 
is a unitary operator. In the following we prove that $\cR$ is a reproducing kernel Hilbert space and its 
reproducing kernel is $K^{(2)}$, that is, $\cR=\cH_{K^{(2)}}$ as Hilbert spaces.

First, we claim that all evaluation functionals on $\cR$ are continuous. Indeed, for arbitrary $x\in X$ and 
$f\in \cR$, as in \eqref{e:lekam} we have
\begin{align*}
|f(x)| &  = \bigl| (L_Kg)(x)\bigr| = \bigl| \int_X K(x,t)g(t)\de\nu(t)\bigr| \\
&  \leq \biggl( \int_X |K(x,t)|^2 \de\nu(t) 
\biggr)^{1/2} \,\|g\|_{L^2_\FF(X,\nu)} =  \biggl( \int_X |K(x,t)|^2 \de\nu(t) 
\biggr)^{1/2} \,\|f\|_{\cR}\\
& = \|K_x\|_{L^2_\FF(X,\nu)}\, \|f\|_{\cR},
\end{align*}
and the claim is proven. Let $M$ denote the reproducing kernel of $\cR$.

By Remark~\ref{r:total}, $\{K_x\mid x\in X\}$ is total in $L^2_\FF(X,\nu)\ominus \ker(L_K)$.
For arbitray $f\in \cR$ and $x\in X$, letting $g\in L^2_\FF(X,\nu)\ominus \ker(L_K)$ such that $f=L_Kg$, we have
\begin{align*}
\langle f,L_KK_x\rangle_\cR & = \langle g,K_x\rangle_{L^2_\FF(X,\nu)} = \int_X g(t) \ol{K_x(t)}\de\nu(t) \\
 & = \int_X K(x,t)g(t)\de\nu(t)=(L_Kg)(x)=f(x),
\end{align*}
hence $M_x=L_K K_x$ for all $x\in X$. Then, for any $x,y\in X$ we have
\begin{align*}
M(x,y) & = M_y(x) = ((L_K K_y)(x)) = \int_X K(x,t)K_y(t)\de\nu(t)\\ & = \int_X K(x,t)K(t,y)\de\nu(t) = K^{(2)}(x,y),
\end{align*}
 which proves that $K^{(2)}$ is the reproducing kernel of the reproducing kernel Hilbert space $\cR$, 
 hence $\ran(L_K)=\cR=\cH_K$ and, when viewing 
$L_K\colon L^2_\FF(X,\nu)\ominus\ker(L_K)\ra \cH_{K^{(2)}}$, it is a unitary operator.

(3) Indeed, to see this, for arbitrary $f\in L^2_\FF(X,\nu)$ we have
\begin{align*}
(L_K^2f)(x) & = (L_K L_Kf)(x) = \int_X K(x,t)(L_Kf)(t)\de\nu(t)\\
&  = \int_X K(x,t)\int_X K(t,s)f(s)\de\nu(s)\de\nu(t)
= \int_X f(s)\bigl(\int_X K(x,t)K(t,s)\de\nu(t)\bigr) \de\nu(s) \\ & = \int_X K^{(2)}(x,s) f(s)\de\nu(s) =(L_{K^{(2)}}f)(x),
\end{align*}
which proves the claim.

(4) Since $\supp(\nu)$ is a closed subset of the compact metric space $X$, it is a compact metric space 
and $L^2_\FF(X,\nu)=L^2_\FF(\supp(\nu),\nu)$. Since $\supp(\nu)\supseteq \supp(K)$, when considering the kernel 
$K\colon \supp(\nu)\times \supp(\nu)\ra\FF$ the corresponding reproducing kernel Hilbert space 
$\cH_K$ remains 
the same. In conclusion, by replacing $X$ with $\supp(\nu)$, without loss of generality we can assume that 
$X=\supp(\nu)$.

By (3), we have
\begin{equation*}
L_{K^{(2)}}=L_K^2=L_K^{1/2} L_KL_K^{1/2}\leq \|L_K\| L_K^{1/2}L_K^{1/2}=\|L_K\| L_K.
\end{equation*}
Since $\supp(\nu)=X$, by Proposition~\ref{p:lekad} and Proposition~\ref{p:conv}, it follows that 
$K^{(2)}\leq \|L_K\| K$ and then, by Proposition~\ref{t:inclusion} it follows that $\cH_{K^{(2)}}\subseteq \cH_K$.
\end{proof}

Next we show that $\cH_K$ is the operator range of $L_K^{1/2}$ and from here we get a more explicit 
representation of $\cH_K$. 

\begin{theorem}\label{t:leka} 
Let $K$ be a Mercer kernel on the compact metric space $X$ and let $\nu$ be a finite Borel 
measure on $X$ such that $\supp(\nu)\supseteq \supp(K)$. 

\nr{i} The reproducing kernel 
Hilbert space $\cH_K$ associated to $K$ is continuously and densely included in 
$L^2_\FF(X,\nu)\ominus \ker(L_K)$, where $L_K\colon L^2_\FF(X,\nu)\ra L^2_\FF(X,\nu)$ is the integral operator with symbol 
$K$ as in \eqref{e:leka}, hence continuously included in $L^2_\FF(X,\nu)$.

\nr{ii} $\ran(L_K)\subseteq \cH_K$ and, 
letting $J_K\colon \cH_K\ra L^2_\FF(X,\nu)$ denote the inclusion operator, when viewing 
$L_K\colon L^2_\FF(X,\nu)\ra \cH_K$, it is a bounded operator and $L_K=J_K^*$ and, when viewing $L_K\colon L^2_\FF(X,\nu)\ra L^2_\FF(X,\nu)$, we have
\begin{equation}\label{e:lekajeka}
L_K=J_KJ_K^*.
\end{equation}

\nr{iii} $\ran(L_K^{1/2})=\cH_K$ and, 
when viewing $L_K^{1/2}\colon L^2_\FF(X,\nu)\ominus \ker(L_K)\ra \cH_K$, it is a unitary operator.
\end{theorem}

\begin{proof} (i) 
Since $\supp(\nu)$ is a closed subset of the compact metric space $X$, it is a compact metric space 
and $L^2_\FF(X,\nu)=L^2_\FF(\supp(\nu),\nu)$. Since $\supp(\nu)\supseteq \supp(K)$, when considering the kernel 
$K\colon \supp(\nu)\times \supp(\nu)\ra\FF$ the corresponding reproducing kernel Hilbert space 
$\cH_K$ remains 
the same. In conclusion, by replacing $X$ with $\supp(\nu)$, without loss of generality we can assume that 
$X=\supp(\nu)$. Then, the Banach space $\cC_\FF(X)$  is continuously included in $L^2_\FF(X,\nu)$ and then, 
since $\cH_K$ is continuously included in $\cC_\FF(X)$, it follows that $\cH_K$ is continuously included in 
$L^2_\FF(X,\nu)$.
 
In Remark~\ref{r:total} it is shown that $\{K_x\mid x\in X\}$ is total in 
$L^2_\FF(X,\nu)\ominus\ker(L_K)$, hence $\cH_K$ is included and dense in 
$L^2_\FF(X,\nu)\ominus\ker(L_K)$. Combining this with the fact proven before, that 
$\cH_K$ is continuously included in $L^2_\FF(X,\nu)$, it follows that $\cH_K$ is continuously included 
in $L^2_\FF(X,\nu)\ominus \ker(L_K)$.

(ii) As before, without loss of generality, we assume that $X=\supp(\nu)$. Let $f=K_x$ and $g=K_y$ for some 
$x,y\in X$. Then, on the one hand,
\begin{equation*}
\langle J_Kf,g\rangle_{L^2}=\langle f,g\rangle_{L^2}=\int_X K_x(t)\ol{K_y(t)}\de\nu(t).
\end{equation*}
On the other hand, taking into account that, by Proposition~\ref{p:kadoi}, we have $\ran(L_K)\subseteq \cH_K$, 
we have
\begin{align*}
\langle f,L_Kg\rangle_{\cH_K} & = \langle K_x,\int_X K(\cdot,t)K(t,y)\de\nu(t)\rangle_{\cH_K}=
\langle K_x,\int_X K_t(\cdot) K(t,y)\de\nu(t)\rangle_{\cH_K} \\
\intertext{which, in view of the reproducing property, equals}
& = \int_X \ol{K(t,y)} K(t,x)\de\nu(t)=\int_X K_x(t)\ol{K_y(t)}\de\nu(t),
\end{align*}
hence,
\begin{equation}\label{e:jeka}
\langle J_K f,g\rangle_{L^2} = \langle f,L_Kg\rangle_{\cH_K}.
\end{equation}
By linearity, the identity in \eqref{e:jeka} is extended for all $f,g\in \Span\{K_x\mid x\in X\}$, which is a subspace
dense in $\cH_K$ with respect to the norm $\|\cdot\|_{\cH_K}$, and dense in $L^2_\FF(X,\nu)\ominus \ker(L_K)$ with
respect to the norm $\|\cdot\|_{L^2}$, by (i). By Proposition~\ref{p:kadoi}, when viewing the operator
$L_K\colon L^2_\FF(X,\nu)\ominus\ker(L_K)\ra \cH_{K^{(2)}}$ it is unitary and $\cH_{K^{(2)}}\subseteq \cH_K$, 
hence by Theorem~\ref{t:inclusion} we have the continuous inclusion 
$\cH_{K^{(2)}}\hookrightarrow \cH_K$, which implies that 
when viewing $L_K\colon L^2_\FF(X,\nu)\ominus\ker(L_K)\ra \cH_K$ it is a bounded operator. From here we get that
the identity \eqref{e:jeka} holds for all $f\in\cH_K$ and all 
$g\in L^2_\FF(X,\nu)\ominus \ker(L_K)$. But, because $\ran(J_K)\subseteq L^2_\FF(X,\nu)\ominus \ker(L_K)$, the 
identity \eqref{e:jeka} is trivial when $g\in \ker(L_K)$, hence it holds for all $f\in \cH_K$ and all $g\in L^2_\FF(X,\nu)$.
This shows that, when considering $J_K\colon \cH_K\ra L^2_\FF(X,\nu)$ and $L_K\colon L^2_\FF(X,\nu)\ra \cH_K$, 
we have $J_K^*=L_K$. Then, since $J_K$ acts like the identity, we have $J_KJ_K^*=L_K$, when considering 
$L_K\colon L^2_\FF(X,\nu)\ra L^2_\FF(X,\nu)$.

(iii) As before, without loss of generality we can assume that 
$\supp(\nu)=X$. Also, from item (ii), $\cH_K$ is
dense in $L^2_\FF(X,\nu)\ominus \ker(L_K)$ and the integral 
operator $L_K\in\cB(L^2_\FF(X,\nu))$, which is positive and a Hilbert-Schmidt operator, see 
Proposition~\ref{p:leka}, can be factored as in \eqref{e:lekajeka}.
 In particular, its 
square root $L_K^{1/2}\in\cB(L^2_\FF(X,\nu))$ is well-defined and a compact operator as well. But, taking into 
account that $L_K^{1/2}$ can be viewed as a positive operator  
$L_K^{1/2}\colon L^2_\FF(X,\nu)\ominus \ker(L_K)\ra L^2_\FF(X,\nu)\ominus \ker(L_K)$ and, similarly, that
$J_K\colon \cH_K\hookrightarrow L^2_\FF(X,\nu)\ominus \ker(L_K)$, the factorisation \eqref{e:lekajeka} yet holds 
and then, from Lemma~\ref{l:pai}, it follows that
\begin{equation}\label{e:ranal}
\ran(L_K^{1/2})=\ran(J_K)=\cH_K
\end{equation}
and there exists a unique unitary operator $V\colon L^2_\FF(X,\nu)\ominus \ker(L_K)\ra \cH_K$ such that
\begin{equation*}
L_K^{1/2}=J_K V.
\end{equation*}
From here, taking into account that  $J_K$ acts on $\cH_K$ like identity and \eqref{e:ranal}, it follows that, 
when viewing the operator $L^{1/2}_K\colon L^2_\FF(X,\nu)\ominus \ker(L_K)\ra\cH_K$, it coincides with $V$, hence
it is a unitary operator.
\end{proof}

\subsection{Mercer's Theorem}

The last step towards Mercer's Theorem is to prove that  the series 
that recovers the kernel $K$ from an arbitrary orthonormal basis of $\cH_K$ converges absolutely and uniformly 
on $X\times X$, to be compared with 
the general  Theorem~\ref{t:on} which guarantees that the convergence of the series is pointwise only.

\begin{theorem}\label{t:ona} 
Let $K$ be a Mercer kernel on the compact metric space $X$ and let $(e_n)_n$ be an 
orthonormal basis of $\cH_K$. Then,
\begin{equation}\label{e:kexyes}
K(x,y)=\sum_{n\in \NN}e_n(x)\ol{e_n(y)},\quad x,y\in X,
\end{equation}
where the convergence of the series is absolute and uniform on $X\times X$.
\end{theorem}

For this, we will need a generalised form of a classical result.

\begin{theorem}[Dini's Theorem] Let $X$ be a compact Hausdorff space and $(f_n)_n$ a sequence of functions 
$f_n\colon X\ra \RR$ subject to the following conditions.
\begin{itemize}
\item[(i)] All functions $f_n$ are continuous.
\item[(ii)] There exists a continuous function $f\colon X\ra\RR$ such that $f_n(x)\ra f(x)$, 
as $n\ra\infty$, for all $x\in X$.
\item[(iii)] The sequence $(f_n)_n$ is monotonic.
\end{itemize}
Then, $f_n\ra f$ as $n\ra\infty$ uniformly on $X$.
\end{theorem}

\begin{proof} Without loss of generality we can assume that $f_{n+1}\leq f_n$ for all $n\in\NN$, otherwise 
change $f_n$ to $-f_n$ for all $n\in\NN$, and that $f_n\ra 0$ pointwise, as $n\ra\infty$, otherwise replace $f_n$ 
by $f_n-f$ for all $n\in\NN$.

With assumptions from before, let $\epsilon>0$ be arbitrary. For each $n\in\NN$ let 
\begin{equation}\label{e:una}
U_n:=\{x\in X\mid f_n(x)<\epsilon\}.
\end{equation}
Since $f_n$ is continuous, $U_n$ is open and, since $f_n\ra 0$ pointwise, 
it follows that $\{U_n\mid n\in\NN\}$ is an open 
cover of $X$, which is compact, hence there exists a finite subcover $\{U_{k_j}\mid j=1,\ldots, m\}$ of $X$. But,
since the sequence of functions $(f_n)_n$ is nonicreasing, we have $U_n\subseteq U_{n+1}$ hence, letting
$N=\max\{k_1,\ldots,k_m\}$ it follows that $X=U_N$. In view of the definition as in \eqref{e:una} and the fact the 
sequence of functions $(f_n)_n$ is nonincreasing, this means that 
for all $n\geq N$ and all $x\in X$ we have $f_n(x)<\epsilon$. In this way, we have proven that $f_n\ra f$, as 
$n\ra\infty$, uniformly on $X$. 
\end{proof}

\begin{proof}[Proof of Theorem~\ref{t:ona}] We already know from Theorem~\ref{t:on} that the convergence in the series \eqref{e:kexyes}
is pointwise on $X\times X$. For each $n\in \NN$ consider the sequence $(S_n)_n$ of kernels on $X$
defined by the remainders of the series \eqref{e:kexyes}
\begin{equation}\label{e:sena}
S_n(x,y):=\sum_{k=n+1}^\infty e_k(x)\ol{e_k(y)},\quad x,y\in X.
\end{equation}
Since,
\begin{equation*}
S_n(x,y)=K(x,y)-\sum_{k=1}^n e_k(x)\ol{e_k(y)},\quad x,y\in X,\ n\in\NN,
\end{equation*}
and $\cH_K$ consists of continuous functions only, see Proposition~\ref{p:mercec}, 
it follows that these kernels are continuous. They are also Hermitian/symmetric and positive semidefinite
because $\cK_\FF(X)$ is a convex cone closed under pointwise convergence. 
In particular, the Schwarz Inequality, see Lemma~\ref{l:schwarz}, holds,
\begin{equation}\label{e:senax}
|S_n(x,y)|^2\leq S_n(x,x) S_n(y,y),\quad x,y\in X,\ n\in\NN.
\end{equation}

In the following we prove that the sequence of functions $f_n(x):=S_n(x,x)$, $x\in X$, $n\in\NN$, uniformly 
converges to $0$. Indeed, we verify that the assumptions of the Dini's Theorem are fulfilled: $X$ is a compact 
metric space, all functions $f_n$ are continuous, $f_{n+1}\leq f_n$ for all $n\in\NN$, and $f_n\ra 0$ as 
$n\ra\infty$, pointwise on $X$. By Dini's Theorem it follows that $f_n\ra 0$ as $n\ra\infty$ uniformly on $X$.

Finally, from \eqref{e:senax}, it follows that $S_n(x,y)\ra 0$ as $n\ra\infty$ uniformly on $X\times X$. But, in view 
of the definition of $S_n$, see \eqref{e:sena}, this means that the convergence in the series \eqref{e:kexyes}
is uniform on $X\times X$.
\end{proof}

\begin{remark}\label{r:second} Theorem~\ref{t:ona} provides a different proof of the fact that, if $\nu$ is a 
finite Borel measure on the compact metric space $X$ and $K$ is a Mercer kernel on $X$, then the integral 
operator $L_K$ is positive, see Proposition~\ref{p:lekad}. To see this, let $f\in L^2_\FF(X,\nu)$ be fixed and let 
$(e_n)_n$ be an orthonormal basis of $\cH_K$. Given $\varepsilon>0$ arbitrary, from Theorem~\ref{t:ona} there exists $N\in\NN$ such that
\begin{equation*}
|K(x,y)-\sum_{k=1}^N e_k(x)\ol{e_k(y)}|<\varepsilon,\quad x,y\in X,
\end{equation*}
hence
\begin{align*}
\bigl| \langle L_K f,f\rangle_{L^2} & 
-\int_X\int_X \sum_{k=1}^n e_k(x)\ol{e_k(y)} \ol{f(x)} f(y)\de\nu(y)\de\nu(x)\bigr| \\
& \leq  \int_X\int_X \bigl|K(x,y) -\sum_{k=1}^n e_k(x)\ol{e_k(y)} \bigr|\, \bigl|\ol{f(x)} f(y)\bigr|\de\nu(y)\de\nu(x)  \\
& \leq \varepsilon\, \biggl( \int_X |f(x)| \de\nu(x)\biggr)^2 \leq \varepsilon \,\|f\|^2_{L^2}\, \nu(X),
\end{align*}
and, consequently, recalling that $\cH_K\subseteq L^2_\FF(X,\nu)$, see Theorem~\ref{t:leka},
we get
\begin{align*}
\langle L_k f,f\rangle_{L^2} & =\lim_{n\ra\infty} \int_X\int_X \sum_{k=1}^n e_k(x)\ol{e_k(y)} \ol{f(x)} f(y)\de\nu(y)\de\nu(x) \\
& = \lim_{n\ra\infty} \sum_{k=1}^n \int_X\int_X
\ol{f(x)} e_k(x)\ol{e_k(y)}f(y)\de\nu(y)\de\nu(x) \\
& =\lim_{n\ra \infty} \sum_{k=1}^n |\langle f,e_k\rangle_{L^2}|^2\geq 0.
\end{align*}
\end{remark}

With assumptions and notation as in Theorem~\ref{t:leka}, the integral operator 
$L_K\in\cB(L^2_\FF(X,\nu))$ is positive and compact, see Proposition~\ref{p:leka} and Proposition~\ref{p:lekad}. Then, as in 
Subsection~\ref{ss:hso}, its spectrum consists in positive eigenvalues of finite 
multiplicities that accumulate at $0$, if they are infinitely many, and probably $0$.  
Then we can write all its nonzero eigenvalues
$\lambda_1\geq \lambda_2\geq \cdots \geq \lambda_n\geq \lambda_{n+1}\geq \cdots >0$, counted with 
their multiplicities, and we can find an orthonormal system of eigenfunctions $\{\phi_n\mid n\in \NN\}$ such that 
$\phi_n$ is an eigenfunction of $L_K$ corresponding to the eigenvalue $\lambda_n$, for each $n\in \NN$. In 
addition, the same Proposition~\ref{p:leka} tells us that all eigenfunctions $\phi_n$ are continuous on $X$, 
because $\phi_n(x)=\frac{1}{\lambda_n} (L_K\phi_n)(x)$, and that $L_K$ is a Hilbert-Schmidt operator and, 
consequently,
\begin{equation}\label{e:lekahs}
\|L_K\|^2_{\mathrm{HS}}=\sum_{n\in\NN} \lambda_n^2<\infty.
\end{equation}

\begin{theorem}[Mercer's Theorem]\label{t:mercer}
Let $K$ be a Mercer kernel on a compact metric space $X$, let $\nu$ be a finite Borel measure on $X$ such that 
$\supp(\nu)\supseteq \supp(K)$, and let $L_K\in\cB_2(L^2_\FF(X,\nu))$ be the integral operator defined by $K$. 
Then,
with notation as above, we have
\begin{equation}\label{e:kexysen}
K(x,y)=\sum_{n=1}^\infty \lambda_n \phi_n(x)\ol{\phi_n(y)},\quad x,y\in X,
\end{equation} 
where the series converges absolutely and uniformly on $X\times X$.
\end{theorem}

\begin{proof} Let us observe that $\{\phi_n\mid n\in\NN\}$ is an orthonormal basis of the Hilbert space 
$L^2_\FF(X,\nu)\ominus \ker(L_K)$. Since, by Theorem~\ref{t:leka}, when viewing $L_K^{1/2}\colon L^2_\FF(X,\nu)\ominus \ker(L_K)\ra\cH_K$ it is a unitary operator, it follows that
\begin{equation*}
\{L_K^{1/2}\phi_n\mid n\in \NN\}=\{\sqrt{\lambda_n}\phi_n\mid n\in \NN\}
\end{equation*}
is an orthonormal basis of $\cH_K$. Then, by Theorem~\ref{t:ona}, it follows that the identity \eqref{e:kexysen} 
holds and the series converges absolutely and uniformly on $X\times X$.
\end{proof}

\begin{corollary} Under the assumptions and with notation as in Theorem~\ref{t:mercer}, the operator $L_K$ is trace-class and 
we have
\begin{equation*}
\tr(L_K)=\sum_{n=1}^\infty \lambda_n=\int_X K(x,x)\de\nu(x).
\end{equation*}
\end{corollary}

\begin{proof} Letting $x=y$ in \eqref{e:kexysen} we have
\begin{equation*}
K(x,x)=\sum_{n=1}^\infty \lambda_n |\phi_n(x)|^2,\quad x\in X,
\end{equation*}
where the series converges absolutely and uniformly on $X$, hence, integrating and changing the order of 
summation with the integral, that is allowed due to the uniform convergence of the series, we get
\begin{align*}
\int_X K(x,x)\de\nu(x) & = \int_X \sum_{n=1}^\infty \lambda_n |\phi_n(x)|^2\de\nu(x) \\
& = \sum_{n=1}^\infty \lambda_n \int_X |\phi_n(x)|^2\de\nu(x) = \sum_{n=1}^\infty \lambda_n,
\end{align*} 
where in the last equality we took into account that the functions $\phi_n$ are normalised in $L^2_\FF(X,\nu)$. In 
particular, this shows that $L_K$ is trace-class, see Subsection~\ref{ss:tco}.
\end{proof}

\begin{corollary}  Under the assumptions and with notation as in Theorem~\ref{t:mercer}, we have
\begin{equation*}
\cH_K=\{f\in L^2_\FF(X,\nu)\mid f=\sum_{n=1}^\infty f_n\phi_n,\ (f_n/\sqrt{\lambda_n})_n\in \ell^2_\FF\},
\end{equation*}
where the series converges absolutely and uniformly on $X$.
\end{corollary}

\begin{proof} By \eqref{e:ranal}, given $f\in L^2_\FF(X,\nu)$, we have $f\in \cH_K$ if and only if there exists 
$g\in L^2_\FF(X\nu)\ominus \ker(L_K)$ unique such that $f=L_K^{1/2}g$. But, $\{\phi_n\mid n\in\NN\}$ is an 
orthonormal basis of $L^2_\FF(X,\nu)\ominus \ker(L_K)$, hence
\begin{equation*}
g=\sum_{n=1}^\infty g_n\phi_n,\quad \sum_{n=1}^\infty |g_n|^2=\|g\|_{L^2_\FF}^2<\infty,
\end{equation*}
hence
\begin{equation}\label{e:felak}
f=L_K^{1/2}g =\sum_{n=1}^\infty g_n\sqrt{\lambda_n}\phi_n=\sum_{n=1}^\infty f_n\phi_n,
\end{equation}
and hence, by the uniqueness of the Fourier represenation, we have $g_n=f_n/\sqrt{\lambda_n}$ for all 
$n\in \NN$ and then
\begin{equation*}
\|f_n\|_{\cH_K}^2=\|g\|_{L^2}^2=\sum_{n=1}^\infty \frac{|f_n|^2}{\lambda_n}<\infty.
\end{equation*}

Finally, the convergence of the series in the extreme right side of \eqref{e:felak} holds in $\cH_K$ but, since the 
inclusion of $\cH_K\hookrightarrow \cC_\FF(X)$ is continuous, see Proposition~\ref{p:mercec}, 
the convergence is uniform on $X$ as well.
\end{proof}

\renewcommand\theequation{A.\arabic{equation}}
\renewcommand\thesection{A}
\renewcommand\thesubsection{A.\arabic{subsection}}

\section{Appendix: Compact, Trace-Class, and Hilbert-Schmidt Operators}\label{s:app}

There are many textbooks and monographs which contain a good presentation on most of the theory of compact 
operator and the Schatten-von Neumann ideals, of which we are interested in the trace-class and 
Hilbert-Schmidt ideals. In writing this section we used \cite{BirmanSolomjak} and \cite{Conway}.

\subsection{Compact Operators.}\label{ss:co}
Given two Hilbert spaces $\cH_1$ and $\cH_2$, a linear operator $T\colon \cH_1\ra \cH_2$ is called 
\emph{compact} if the image of 
the closed unit ball $\{x\in \cH_1\mid \|x\|\leq 1\}$ of $\cH_1$ under $T$ is relatively compact in $\cH_2$, that is,
the closure of the set $\{Tx\mid x\in \cH_1,\ \|x\|\leq 1\}$, is compact in $\cH_2$. Clearly, any compact operator
is continuous.

\begin{proposition} Let $T\in \cB(\cH_1,\cH_2)$, for some Hilbert spaces $\cH_1$ and $\cH_2$. 
The following assertions are equivalent.
\begin{itemize}
\item[(i)] $T$ is compact.
\item[(ii)] The set $\{Tx\mid x\in \cH_1,\ \|x\|\leq 1\}$ is compact in $\cH_2$.
\item[(iii)] For any bounded sequence $(x_n)_n$ in $\cH_1$ there exists a subsequence $(x_{k_n})_n$ 
such that $(Tx_{k_n})_n$ converges in the norm topology of $\cH_2$.
\item[(iv)] For any sequence $(x_n)_n$ in $\cH_1$ that converges weakly to $0\in \cH_1$, 
the sequence $(\|Tx_n\|)_n$ converges to $0$.
\item[(v)] The adjoint operator $T^*\in \cB(\cH_2,\cH_1)$ is compact.
\item[(vi)] The modulus $|T|\in \cB(\cH_1)$ is compact. 
\item[(vii)] There exists a sequence $(T_n)_n$ of compact operators, $T_n\colon \cH_1\ra \cH_2$ for all 
$n\geq 1$, such that $(T_n)_n$ converges to $T$ with respect to 
the operator norm in $\cB(\cH_1,\cH_2)$, that is, $\|T_n-T\|\xrightarrow[n\ra\infty]{}0$.
\end{itemize}
\end{proposition}

Let $\cB_0(\cH_1,\cH_2)$ denote the collection of all compact operators $T\colon \cH_1\ra \cH_2$. Clearly, 
$\cB_0(\cH_1,\cH_2)$ is a vector subspace of $\cB(\cH_1,\cH_2)$ and, as a consequence of the item (vii) in the
previous proposition, it follows that it is closed with respect to the uniform topology (the topology induced by
the operator norm).

Given two vectors $x\in \cH_1$ and $y\in \cH_2$, the linear operator 
$x\otimes \overline y\colon \cH_1\ra \cH_2$
defined by $(x\otimes \overline y)z:= \langle z,x\rangle y$, for all $z\in \cH_1$, is linear
and its range is spanned by the vector $y$, in particular it is bounded. 
More generally, a linear operator $T\colon \cH_1\ra \cH_2$ has
\emph{finite rank} if the range of $T$, $\ran(T)=T\cH_1\subseteq \cH_2$, is finite dimensional. Clearly, any finite
rank operator is bounded. Also, it is easy to see that, for any operator $T\in \cB(\cH_1,\cH_2)$ with finite rank,
there exist $x_1,\ldots,x_n\in \cH_1$ and $y_1,\ldots,y_n\in \cH_2$ such that
\begin{equation}\label{e:tesuj}
T=\sum_{j=1}^n x_j\otimes \overline y_j.
\end{equation}

Let $\cB_{00}(\cH_1,\cH_2)$ denote the collection
of all finite rank operators $T\colon \cH_1\ra \cH_2$. Clearly, $\cB_{00}(\cH_1,\cH_2)$ is a vector subspace of 
$\cB_0(\cH_1,\cH_2)$. The decomposition \eqref{e:tesuj} can be put in 
the following equivalent form: 
\begin{equation}\label{e:tesujon}
T=\sum_{j=1}^n \lambda_n x_n\otimes\overline y_n,
\end{equation}
where $\{x_1,\ldots,x_n\}$ and $\{y_1,\ldots,y_n\}$ are orthonormal in $\cH_1$ and, respectively, in $\cH_2$, 
and $\lambda_1,\ldots,\lambda_n$ are nonzero scalars. Then, it has the following generalisation to
compact operators that are not of finite rank.

\begin{proposition} With notation as before, let $T\in \cB(\cH_1,\cH_2)$. The following assertions are equivalent.
\begin{itemize}
\item[(i)] $T$ is compact and with infinite rank.
\item[(ii)] There exist $(x_n)_n$ an orthonormal sequence in $\cH_1$, an orthonormal sequence $(y_n)_n$
in $\cH_2$, and a sequence of nontrivial scalars $(\lambda_n)_n$ such that
\begin{equation}\label{e:schmidt}
T=\sum_{n=1}^\infty \lambda_n x_n\otimes \overline y_n,
\end{equation}
where the series converges in the operator norm.
\end{itemize}
In particular, $\cB_{00}(\cH_1,\cH_2)$ is dense in $\cB_0(\cH_1,\cH_2)$, with respect to the
uniform topology.
\end{proposition}

The decomposition of compact operators as in \eqref{e:schmidt} or \eqref{e:tesujon} is called the 
\emph{Schmidt decomposition}
of $T$.

If $\cH$ is a Hilbert
space then we denote $\cB_0(\cH):=\cB_0(\cH,\cH)$ and $\cB_{00}(\cH):=\cB_{00}(\cH,\cH)$. 
Then, $\cB_{00}(\cH)\subseteq \cB_0(\cH)$ are 
two sided ideals of $\cB(\cH)$, stable under taking the adjoint, and $\cB_{00}(\cH)$ is dense in $\cB_0(\cH)$. 
Also, $\cB_0(\cH)$ is a $C^*$-subalgebra of $\cB(\cH)$ and it does not have a unit, unless $\cH$ is finite dimensional.

From the point of view of spectral theory, we have the following result. As usually, given $T\in\cB(\cH)$ for some 
Hilbert space $\cH$, we denote by $\sigma(T)$ the \emph{spectrum} of $T$, 
that is, $\sigma(T):=\{\lambda \in \CC\mid \lambda I-T\mbox{ is not invertible in }\cB(\cH)\}$.
\begin{proposition}
Let $T\in \cB_0(\cH)$ for some Hilbert space $\cH$.
\begin{itemize}
\item[(a)] $\sigma(T)$ is countable and, if infinite, it accumulates only at $0$.
\item[(b)] Any $\lambda\in\sigma(T)\setminus\{0\}$ is an eigenvalue of $T$ of finite multiplicity.
\end{itemize}
\end{proposition}

For the case of a compact normal operators, the Schmidt decomposition provides a concrete form of
the spectral measure. To see this, let us first observe that any 
projection of rank one in $\cH$ is of the form $x\otimes\overline x$, where $x\in \cH$ with $\|x\|=1$.

\begin{proposition} Let $T\in \cB(\cH)$ be a nontrivial compact normal operator. Then, there exists an 
orthonormal 
sequence $(x_n)_{n=1}^N$, where $N\in \NN\cup\{\infty\}$, 
and a sequence of scalars $(\lambda_n)_{n=1}^N$, with $\lambda_n\xrightarrow[n\ra\infty]{} 0$ if $N=\infty$,
such that 
\begin{equation*}
T=\sum_{n=1}^N \lambda_n x_n\otimes\overline x_n.
\end{equation*}
In particular, $\sigma(T)\subseteq \{\lambda_n\mid n=1,\dots,N\}\cup\{0\}$, with equality if either $\cH$ is infinite
dimensional or $\cH$ is finite dimensional and $T$ is not invertible, and for each $n=1,\ldots,N$, $x_n$ is an 
eigenvector corresponding to $T$ and the eigenvalue $\lambda_n$.
\end{proposition}

Given a compact operator $T\in\cB(\cH,\cK)$, the operator $|T|=(T^*T)^{1/2}$ is a compact positive operator and 
hence its spectrum consists of a finite or infinite sequence of nonzero eigenvalues, all of them positive, and 
probably $0$. 
Let this sequence of positive eigenvalues, called the sequence of
\emph{singular numbers}  of $T$, 
be denoted by $(s_n(T))_{n=1}^N$, counted with their multiplicities, 
where $N\in\NN$ or $N=\infty$, 
with $s_{n+1}(T)\leq s_n(T)$ for all $n=1,\ldots,N$.  $T$ is a finite rank operator if and only 
if $N\in\NN$. If $T$ is not of finite rank then $s_n(T)\xrightarrow[n\ra\infty]{}0$. Let $(\phi_n)_n$ 
be an orthonormal sequence of eigenvectors of $|T|$ such that $|T|\phi_n=s_n(T)\phi_n$ for all $n$. 
Anyway, $\{\phi_n\}_n$ is an orthonormal basis of $\cH\ominus\ker(T)$.

One can 
reformulate the Schmidt decomposition  \eqref{e:schmidt} of a compact operator in a more precise fashion. Let 
$T=V|T|$ be the \emph{polar decomposition} of $T$, where $V\in\cB(\cH,\cK)$ is the unique partial isometry such 
that $V^*V$ is the orthogonal projection onto $\cH\ominus\ker(T)$ and $VV^*$ is the orthogonal projection onto
$\cK\ominus\ker(T^*)$. Let $\psi_n=V\phi_n$ for all $n$. Then, for all $n$, we have
\begin{equation}\label{e:tefa}
T\phi_n=s_n(T)\psi_n,\quad T^*\psi_n=s_n(T)\phi_n,\quad T^*T\phi_n
=s_n^2(T)\phi_n,\quad TT^*\psi_n=s_n(T)^2 \psi_n.
\end{equation}

\begin{theorem} Given $T\in\cB_0(\cH)$, with notation as before, the following assertions hold true.

\nr{1} The operator $T$ has the expansion
\begin{equation}\label{e:schmidto}
T=\sum_{n=1}^N s_n(T)\psi_n\otimes\ol{\phi_n},
\end{equation}
where the convergence holds with respect to the operator norm of $\cB(\cH)$.

\nr{2} The following equalities hold,
\begin{equation*}
T^*= \sum_{n=1}^N s_n(T)\phi_n\otimes\ol{\psi_n},
\quad T^*T=\sum_{n=1}^N s_n^2(T)\phi_n\otimes\ol{\phi_n}, \quad
TT^*= \sum_{n=1}^N s_n^2(T)\psi_n\otimes\ol{\psi_n},
\end{equation*}
where the convergence of the series holds with respect to the operator norm of $\cB(\cH)$.
\end{theorem}

\subsection{Trace-Class Operators}\label{ss:tco}
 An operator $T\in\cB_0(\cH)$ is called of \emph{trace-class} or \emph{nuclear} if
\begin{equation}\label{e:trc}
\|T\|_1=\sum_{n=1}^\infty s_n(T)<\infty.
\end{equation}
We denote by $\cB_1(\cH)$ the collection of all trace-class operators on $\cH$.

\begin{theorem} $\cB_1(\cH)$ is a Banach space with the norm $\|\cdot\|_1$ defined at \eqref{e:trc}. In addition, 
$\cB_1(\cH)$ is a two-sided $*$-ideal of $\cB(H)$, for all $T\in \cB_1(\cH)$ we have $\|T\|\leq \|T\|_1$, 
$\|T^*\|_1=\|T\|_1$, and $\|ATB\|_1\leq \|A\| \|T\|_1 \|B\|$ for all $A,B\in\cB(\cH)$.  
\end{theorem}

The trace-class operators can be characterised within $\cB(\cH)$ by means of orthonormal bases.

\begin{theorem}
Let $T\in \cB(\cH)$. The following assertions are equivalent.
\begin{itemize}
\item[(i)] For some, equivalently, for any, orthonormal basis $(h_j)_{j\in\cJ}$ we have
\begin{equation*}
\sum_{j\in\cJ} \|Th_j\|<\infty.
\end{equation*}
\item[(ii)] For some, equivalently, for any, orthonormal basis $(h_j)_{j\in\cJ}$ we have
\begin{equation*}
\sum_{j\in\cJ} \langle |T|h_j,h_j\rangle_\cH<\infty.
\end{equation*}
In this case, we have
\begin{equation*}
\sum_{j\in\cJ} \langle |T|h_j,h_j\rangle_\cH=\sum_{n=1}^\infty s_n(T)=\|T\|_1,
\end{equation*}
in particular, the sum of the series does not depend on the orthonormal basis.
\item[(iii)] For any orthonormal bases $(h_j)_{j\in\cJ}$ and $(g_j)_{j\in\cJ}$ of $\cH$, 
the series $\sum_{j\in\cJ} \langle Th_j,g_j\rangle_\cH$ converges.
\end{itemize}
\end{theorem}

The $*$-ideal $\cB_1(\cH)$ is closely related to the concept of the \emph{trace} of an operator.
\begin{theorem}
Let $T\in\cB_1(T)$. Then, for any orthonormal basis $\{h_j\}_{j\in\cJ}$ of $\cH$, the sum
\begin{equation*}
\tr(T):=\sum_{j\in\cJ} \langle Th_j,h_j\rangle_\cH
\end{equation*}
converges and does not depend on the orthonormal basis. In particular, with notation as 
in \eqref{e:tefa}, we have
\begin{equation}
\tr(T)=\sum_{n=1}^\infty s_n(T) \langle \psi_n,\phi_n\rangle_\cH
=\sum_{n=1}^\infty \langle T\phi_n,\phi_n\rangle_\cH,\end{equation}
and, if $T\geq 0$ then
\begin{equation*}
\tr(T)=\sum_{n=1}^\infty s_n(T)=\|T\|_1.
\end{equation*}
\end{theorem}

\subsection{Hilbert-Schmidt Operators} \label{ss:hso}
An operator $T\in\cB_0(\cH)$ is \emph{Hilbert-Schmidt} if
\begin{equation*}
\|T\|_2^2:=\sum_{n=1}^\infty s_n^2(T)<\infty.
\end{equation*}
We denote by $\cB_2(\cH)$ the collection of all Hilbert-Schmidt operators on the Hilbert space $\cH$.

\begin{theorem} Let $T\in\cB(\cH)$ for some Hilbert space $\cH$. Then $T$ is a Hilbert-Schmidt operator 
if and only if for some, equivalently, for any orthonormal basis $\{h_j\}_{j\in\cJ}$ of $\cH$, we have
\begin{equation*}
\sum_{j\in\cJ} \|Th_j\|_\cH^2<\infty.
\end{equation*}
In this case,
\begin{equation*}
\|T\|_2^2=\sum_{j\in\cJ} \|Th_j\|_\cH^2,
\end{equation*}
in particular, the sum does not depend on the orthonormal basis.
\end{theorem}

The following theorem gathers most of the useful properties of the Hilbert-Schmidt operators.

\begin{theorem} \emph{(1)} $\cB_2(\cH)$ is a two-sided $*$-ideal of $\cB(\cH)$, $\|T\|\leq \|T\|_2$ for all 
$T\in\cB_2(\cH)$, and $\|ATB\|_2\leq \|A\| \|T\|_2 \|B\|$ for all $A,B\in \cB(\cH)$ and all $T\in\cB_2(\cH)$.

\nr{2} For any $S,T\in\cB_2(\cH)$ we have $ST\in\cB_1(\cH)$ and $\|ST\|_1\leq \|S\|_2 \|T\|_2$. Moreover,
letting
\begin{equation*}
\langle S,T\rangle_{\mathrm{HS}}:=\tr(T^*S),\quad S,T\in\cB_2(\cH),
\end{equation*}
we get a Hilbert space $(\cB_2(\cH),\langle\cdot,\cdot\rangle_{\mathrm{HS}})$ such that 
$\langle T,T\rangle_{\mathrm{HS}}=\|T\|_2^2$ for all $T\in\cB_2(\cH)$.

\nr{3} Let $\{h_j\}_{j\in\cJ}$ and $\{g_j\}_{j\in\cJ}$ be two orthonormal bases of 
$(\cH,\langle\cdot,\cdot\rangle_\cH)$. Then $\{h_j\otimes\ol{g_k}\}_{j,k\in\cJ}$ is an orthonormal bases of 
$(\cB_2(\cH),\langle\cdot,\cdot\rangle_{\mathrm{HS}})$.
\end{theorem}

\end{document}